\documentclass[review,hidelinks,onefignum,onetabnum]{siamart220329}

\usepackage{amsfonts}
\usepackage{amsmath,version}
\usepackage{amssymb,graphicx,fancybox,mathrsfs,stmaryrd,relsize}
\usepackage{url,hyperref}
\usepackage{subfigure}
\usepackage{color}
\usepackage{cases}

\usepackage{subeqnarray} 
\usepackage{cases}

\usepackage{algorithm}  
\usepackage{algorithmic}  

\newtheorem{exm}{\bf Example} 
\newenvironment{example}{\begin{exm}} {\end{exm}} 

\newcommand{\bs}[1]{\boldsymbol{#1}}
\def \ri {{\rm i}}

\newsiamremark{rem}{Remark}
\newsiamremark{prop}{Proposition}
\newsiamremark{lm}{Lemma}
\newsiamremark{thm}{Theorem}

\usepackage{lipsum}
\usepackage{amsfonts}
\usepackage{graphicx}
\usepackage{epstopdf}
\usepackage{algorithmic}
\ifpdf
  \DeclareGraphicsExtensions{.eps,.pdf,.png,.jpg}
\else
  \DeclareGraphicsExtensions{.eps}
\fi

\newsiamremark{remark}{Remark}
\newsiamremark{hypothesis}{Hypothesis}
\crefname{hypothesis}{Hypothesis}{Hypotheses}
\newsiamthm{claim}{Claim}

\headers{Divergence-free spectral algorithm for curl-curl equation}{L. Qin, C. Sheng, and Z. Yang}

\title{A highly efficient and accurate divergence-free spectral method for curl-curl equation in two and three dimensions\thanks{Submitted to the editors DATE.}}

\author{Lechang Qin\thanks{School of Mathematical Sciences,  Shanghai Jiao Tong University, Shanghai, 200240, China. 
  (\email{tony$\_$qin0225@sjtu.edu.cn}  (L. Qin)).}
\and Changtao Sheng\thanks{School of Mathematics, Shanghai University of Finance and Economics, Shanghai 200433, China.  The research of this author is partially supported by by the National Natural Science Foundation of China (Nos. 12201385 and 12271365), Shanghai Pujiang Program 21PJ1403500, and the Fundamental Research Funds for the Central Universities 2021110474. 
  (\email{ctsheng@sufe.edu.cn} (C. Sheng)).}
\and Zhiguo Yang\thanks{Corresponding author. School of Mathematical Sciences, MOE-LSC and CMA-Shanghai, Shanghai Jiao Tong University, Shanghai 200240, China. The research of this author is partially supported by the National Natural Science Foundation of China (No. 12101399), the Shanghai Sailing Program (No. 21YF1421000), the Strategic Priority Research Program of Chinese Academy of Sciences (No. XDA25010402) and the Key Laboratory of Scientific and Engineering Computing (Ministry of Education).
  (\email{yangzhiguo@sjtu.edu.edu} (Z. Yang)).}}

\usepackage{amsopn}

\ifpdf
\hypersetup{
  pdftitle={An Example Article},
  pdfauthor={D. Doe, P. T. Frank, and J. E. Smith}
}
\fi

\graphicspath{{../figs/}}

\begin{document}
\nolinenumbers
\maketitle

\begin{abstract}
In this paper, we present a fast divergence-free spectral algorithm (FDSA) for the curl-curl problem. Divergence-free bases in two and three dimensions are constructed by using the generalized Jacobi polynomials.  An accurate spectral method with exact preservation of the divergence-free constraint point-wisely is then proposed, and its corresponding error estimate is established.  We then present a highly efficient solution algorithm based on a combination of matrix-free preconditioned Krylov subspace iterative method and a fully diagonalizable auxiliary problem, which is derived from the spectral discretisations of generalized eigenvalue problems of Laplace and biharmonic operators. We rigorously prove that the dimensions of the invariant subspace of the preconditioned linear system resulting from the divergence-free spectral method with respect to the dominate eigenvalue $1$, are $(N-3)^2$ and $2(N-3)^3$ for two- and three-dimensional problems with $(N-1)^2$ and $2(N-1)^3$ unknowns, respectively. Thus, the proposed method usually takes only several iterations to converge, and astonishingly,  as the problem size (polynomial order) increases, the number of iterations will decrease, even for highly indefinite system and oscillatory solutions. As a result, the computational cost of the solution algorithm is only a small multiple of $N^3$ and $N^4$ floating number operations for 2D and 3D problems, respectively. Plenty of numerical examples for solving the curl-curl problem with both constant and variable coefficients in two and three dimensions are  presented to demonstrate the accuracy and efficiency of the proposed method. 
\end{abstract}



\section{Introduction}
The curl-curl problem with a divergence-free constraint arises naturally in many real-world applications related to electromagnetics, such as semiconductor manufacturing, astrophysics, inertial confinement fusion, photonic crystals, and more \cite{chen1984introduction, davidson2002introduction, joannopoulos2011photonic, markowich2012semiconductor,zhang2020double}. When dealing with problems of this type, which involve variable coefficients, inhomogeneities, or nonlinear couplings, obtaining analytical solutions can often be challenging.  In such scenarios, efficient and accurate simulations of these problems are of great importance in gaining a deep understanding of the underlying physical mechanisms and in addressing practical engineering problems.

Nevertheless, the aforementioned problem is notoriously difficult to solve numerically. It poses challenges and difficulties that can be attributed to at least three different aspects. Firstly, even a small violation of the divergence-free constraint can lead to a significant deviation of the numerical solution from the actual solution \cite{brackbill1985fluid, brackbill1980effect}. Secondly, the indefinite nature of the system can cause the solution algorithm to converge slowly \cite{boffi2013mixed, Peter2003}. Finally, an excessively large number of degrees of freedom may be required to accurately approximate highly oscillatory solutions \cite{babuska1997pollution, gottlieb1977numerical}. On the one hand, it is essential to construct approximation spaces that ensure well-posedness of the discrete problems and adherence to the divergence-free constraint (either strongly or weakly). On the other hand, it is equally important to devise efficient algorithms for solving the resulting indefinite linear system, particularly for solutions with high oscillation.

There has been a long standing interest in developing numerical algorithms for the curl-curl problem that preserve the divergence-free constraint at the discrete level. The common methods of choice include, but not limited to, staggered-grid finite difference Yee algorithm \cite{ taflove1980application, yee1966numerical}, the vector potential method \cite{cyr2013new, jardin2010computational,shadid2010towards}, the divergence-regularzation method \cite{costabel2002weighted, duan2012delta,hazard1996solution}, and the Lagrange multiplier method \cite{Peter2003}. These methods require either the solution of an extra equation to project the numerical solution to certain divergence-free space, or augmenting/reformulating the original system with new variables, leading to equations with higher spatial order or large coupled system to be solved. In contrast to the aforementioned methods,  the method of constructing divergence-free basis can not only preserve the divergence-free constraint in a straightforward way, but also significantly reduces the degrees of freedom to be solved.  In \cite{ye1997discrete}, a weakly divergence-free finite element basis was constructed for the Stokes equation. In \cite{guo2016spectral}, Guo and Jiao proposed a divergence-free basis functions for $n$-dimensional Navier-Stokes equation. Cai et. al in \cite{cai2013divergence} constructed a divergence-free $\bs H({\rm div})$ hierarchical basis for the MHD system. Local divergence-free methods based on discontinuous-Garlerkin (DG) or weak-Galerkin (WG) frameworks were also proposed in \cite{hiptmair2018fully, li2005locally, mu2018discrete}.  However, to the best of the authors' knowledge, there are no fast divergence-free spectral algorithm available for the curl-curl problem.

The main contributions of this paper are twofold. Firstly, we derive a novel divergence-free spectral approximation  bases using generalized Jacobi polynomials for the curl-curl problem in both two and three dimensions. We conduct a rigorous error estimate to justify the spectral accuracy of this method. Secondly, we propose a highly efficient solution algorithm, which draws inspirations from spectral discretizations of the eigenvalue problems of Laplace and biharmonic operators and the fast diagonalisation method for Poisson-type equations. Lemma \ref{prop: MS}  demonstrates that the mass and stiffness matrices resulted from the proposed method are closely related, inspiring the construction of a fast diagonalizable auxiliary problem. A highly efficient solution algorithm is then proposed by a proper combination of the auxiliary problem and matrix-free preconditioned Krylov subspace iterative method. We prove rigorously that the dimensions of the invariant subspace of the preconditioned system with respect to eigenvalue 1 are precisely $(N-3)^2$ and $2(N-3)^3$ for two- and three-dimensional problems with $(N-1)^2$ and $2(N-1)^3$  unknowns,  respectively. As a result, the proposed fast divergence-free spectral algorithm usually takes only several iterations to converge, and as the problem size increases, the number of iterations will decrease, even for highly indefinite system and oscillatory solutions. Thanks to these properties, the presented divergence-free spectral method and corresponding fast solution algorithms are very competitive and computationally attractive.

The remainder of the paper is structured as follows: In Section 2, we introduce the divergence-free spectral method and conduct the corresponding error estimate for the curl-curl problem in two dimensions. Then, we present a highly efficient solution algorithm that combines a fast diagonalizable auxiliary problem and matrix-free preconditioned Krylov subspace iterative method. Theoretical results on the dimensions of the invariant subspace of the preconditioned system with respect to the dominate eigenvalue 1 are also presented. In Section 3, we expand upon the proposed divergence-free spectral method and fast solution algorithm to the three-dimensional case. In Section 4, we present ample numerical results to demonstrate the efficiency and accuracy of the proposed method, particularly for indefinite systems and highly oscillatory solutions. Finally, in Section 5, we conclude the discussion with closing remarks.

 \section{A fast divergence-free spectral algorithm for 2D curl-curl problem}

We first introduce some notation and review some relevant properties of the generalised Jacobi polynomials (cf.\,\cite{GUO20091011,ShenTangWang2011}), which will be an indispensable building block for the construction of divergence-free basis. 

\subsection{Generalised Jacobi polynomials} 
For $\alpha,\beta>-1$, the classical Jacobi polynomials $P_{n}^{(\alpha,\beta)}(\xi)$ are mutually orthogonal with respect to the weight function $\omega^{\alpha, \beta}(\xi)=(1-\xi)^{\alpha}(1+\xi)^{\beta}$ on $\Lambda=(-1,1)$, that is
\begin{equation}\label{GUPeq9}
\int_{-1}^1P_n^{(\alpha,\beta)}P_m^{(\alpha,\beta)}\omega^{\alpha, \beta}\,\mathrm{d}\xi=\frac{2^{\alpha+\beta+1}\Gamma(n+\alpha+1)\Gamma(n+\beta+1)}{(2n+\alpha+\beta+1)\Gamma(n+1)\Gamma(n+\alpha+\beta+1)}\delta_{nm},
\end{equation}
where $\delta_{nm}$ is the Kronecker symbol. The classical Jacobi polynomials can be generalized to cases with general $\alpha,\beta\in\mathbb{R}$, that is, the generalised Jacobi functions (cf.\,\cite{GUO20091011, ShenTangWang2011}). The flexibility afforded by the parameters $\alpha,\beta$ allows us to design a suitable basis for exact preservation of the divergence-free constraint.  Hereafter, we adopt the following generalized Jacobi polynomials:
\begin{align}
&P_n^{(0,0)}(\xi)=L_n(\xi), \;\;\xi\in\Lambda, \;n\geq 0, \label{eq: K0}\\
&
P_n^{(-1,-1)}(\xi)=\frac{\xi^2-1}{4}P_{n-2}^{(1,1)}(\xi)=\frac{n-1}{2(2n-1)} \big(L_n(\xi)-L_{n-2}(\xi) \big), \;\;n \geq 2, \label{eq: K1}
\end{align}
where $L_n(\xi)$ also known as the Legendre polynomial of order $n$. In view of \eqref{GUPeq9}, these polynomials satisfy the following orthogonality properties:
\begin{align}
&\int_{-1}^{1}L_m(\xi)L_n(\xi)\,{\rm d}\xi= \frac{2}{2n+1}\delta_{mn}, \label{eq: Linner}\\ 
& \int_{-1}^{1}P_m^{(-1,-1)}(\xi)P_n^{(-1,-1)}(\xi)(1-\xi^2)^{-1}\,{\rm d}\xi=\frac{n-1}{2n(2n-1)} \delta_{mn}. \label{eq: J11inner}
\end{align}
One prominent feature of the generalized Jacobi polynomials is the derivative property
\begin{equation}\label{eq: derip}
 \big(P_n^{(-1,-1)}\big)'(\xi)=\frac{n-1}{2}L_{n-1}(\xi), \;\; n\geq 2,
\end{equation}
which plays a significant role in constructing the high-order divergence-free spectral bases. 

 \subsection{The divergence-free approximation scheme and its error estimate}
We are concerned with the following curl-curl problem in two and three dimensions
 \begin{subequations}\label{eq:system}
\begin{align}
&\nabla\times\nabla\times {\bs u}+\kappa \, {\bs u}={\bs f}\;\; {\rm in}\;\;\Omega\subset \mathbb{R}^d, \label{eq: 2dsys1}\\
& {\bs n}\cdot \bs u=0,\quad \bs n\times (\nabla\times {\bs u})=0\;\; {\rm on}\;\;\partial\Omega, \label{eq: 2dsys1bc}
\end{align}
\end{subequations}
where $d=2,3$, $\bs u(\bs x)$ usually represents the magnetic field, $\kappa$ is a constant, $\bs f(\bs x)$ is the source term satisfying $\nabla \cdot \bs f=0$, and $\bs n$ is the unit outward vector normal to the boundary $\partial \Omega$. $\nabla \times $ and $\nabla \cdot $ are the usual curl and divergence operators, respectively. 
%
%
Define $\mathbb{X}^{{\rm div}0}(\Omega):=\big \{ \bs u\in  H({\rm curl};\Omega)\cap H_0({\rm div},\Omega),\; \nabla \cdot \bs u=0 \big \}$, where the Sobolev space $H({\rm curl};\Omega)$ and $H_0({\rm div};\Omega)$ are 
\begin{equation*}\label{eq: sob}
\begin{aligned}
&H({\rm curl};\Omega)=\big \{ \bs u\in (L^2(\Omega))^d, \nabla \times \bs u\in (L^2(\Omega))^d  \big \},\;\;
\\
& H_0({\rm div};\Omega)=\big \{\bs u\in (L^2(\Omega))^d, \nabla \cdot \bs u\in L^2(\Omega), \;\; \bs n\cdot \bs u=0 \;{\rm at}\;\partial \Omega  \big \}.
\end{aligned}
\end{equation*}

The weak formulation of equation \eqref{eq:system} reads: find $\bs u \in \mathbb{ X}^{{\rm div}0}(\Omega)$ such that
\begin{equation}\label{eq: weakform}
(\nabla\times \bs u,\nabla\times \bs v)+\kappa (\bs u,\bs v)=(\bs f,\bs v),\quad \forall \bs v\in \mathbb{ X}^{{\rm div}0}(\Omega).
\end{equation}
In what follows, we let $\Omega:=\Lambda^2=(-1,1)^2$. We propose the divergence-free approximation space in two dimensions as follows.
\begin{prop}
$\mathbb{ X}_{N,2d}^{{\rm div}0}(\Lambda^2)$ is a divergence-free approximation space for $\mathbb{ X}^{{\rm div}0}(\Lambda^2)$ taking the form
\begin{equation}\label{eq: div0b2d}
\mathbb{ X}_{N,2d}^{{\rm div}0}(\Lambda^2)={\rm span} \big \{\bs\Phi_{m,n},\; 1\leq m,n\leq N-1\big \},
\end{equation}
where $ {\bs \Phi}_{m,n}(\bs x)=\big( \psi_{m+1}(x_1)\phi_n(x_2),-\phi_m(x_1)\psi_{n+1}(x_2)    \big)^{\intercal}$ and 
\begin{equation}\label{eq: psiphi}
\begin{aligned}
&\psi_{m+1}(\xi)=\frac{\sqrt{2(2m+1)}}{m}P_{m+1}^{(-1,-1)}(\xi),\quad \phi_m(\xi)=\sqrt{\frac{2m+1}{2}}L_m(\xi).
\end{aligned}
\end{equation}
\end{prop}
\begin{proof}
With the help of  the derivative property \eqref{eq: derip}, one readily verifies that
\begin{equation}\label{eq:psirelaphi}
\psi_{m+1}'(\xi)=\phi_m(\xi),
\end{equation}
from where we obtain $\nabla \cdot \bs\Phi_{m,n}(\bs x)=0.$
\end{proof}
Correspondingly, the approximation scheme for the weak formulation \eqref{eq: weakform} in two dimensions takes the form: find $\bs u_N \in \mathbb{ X}_{N,2d}^{{\rm div}0}(\Lambda^2)$ such that
\begin{equation}\label{eq: weakformdist}
\mathcal{A}(\bs u_N,\bs v):=(\nabla\times {\bs u}_N,\nabla\times {\bs v})+\kappa ({\bs u}_N,\bs v)=({\bs f},\bs v),\quad \forall \bs v\in \mathbb{ X}_{N,2d}^{{\rm div}0}(\Lambda^2).
\end{equation}

In order to conduct the error estimate for the proposed scheme, let us introduce some notation. Let $\mathbb{N}$ be the set of all positive integers, and
let $\mathbb{N}_0=\mathbb{N}\cup \{0\}$. We denote  the $l^1$-norm of
$\bs{\xi}$ in $\mathbb{R}^d$ by $|\bs{\xi}|_1$, and the set of all algebraic polynomials of degree at most $N$ by $\mathcal{P}_N$.  For the sake of simplicity, we also assume $\kappa>0$. Wavenumber-explicit priori estimate and error analysis for curl-curl problem with large negative $\kappa$ can be conducted by Rellich identities, but much more involved. We refer the readers to relevant works for Helmholtz equation (\cite{cummings2006sharp, shen2005spectral, shen2007analysis}) and time-harmonic Maxwell's equations (\cite{hiptmair2011stability, 10.1093/imanum/drx014}).
 
 The bilinear form \eqref{eq: weakform} satisfies the following continuity and coercivity: 
\begin{equation}
\label{Aaoper}\begin{split}
&\mathcal{A}(\bs u,\bs v)
 \leq C\|\bs u\|_{H({\rm curl};\Omega)}\|\bs v\|_{H({\rm curl};\Omega)},\quad  {\rm for}\quad \bs u,\bs v\in \mathbb{X}^{{\rm div}0}(\Omega),\\
&\mathcal{A}(\bs u,\bs u)
 \ge  \alpha\|\bs u\|^2_{H({\rm curl};\Omega)},\qquad \qquad  {\rm for}\quad \bs u\in \mathbb{X}^{{\rm div}0}(\Omega),
\end{split}\end{equation}
where the positive constants $C$ and $\alpha$ are independent of $\bs u.$

By the Helmholtz-Hodge decomposition theory, $\bs u$ can be decomposed  into two components as
\begin{equation}
\bs u =\nabla \times \phi+\nabla \psi,
\end{equation}
which leads to 
\begin{equation}\label{boundary2}
\nabla\cdot\bs u =\Delta \psi,\;\;\;\nabla\times\bs u=\nabla\times\nabla\times\phi.
\end{equation}
By the divergence-free condition, we obtain that $\bs u =\nabla \times \phi.$
We find from \eqref{eq:psirelaphi} that
\begin{equation}\label{anal1}\begin{split}
&\nabla\times\big( \psi_{m+1}(x_1)\phi_n(x_2),-\phi_m(x_1)\psi_{n+1}(x_2)    \big)^{\intercal}
\\&=-\phi_m^\prime(x_1)\psi_{n+1}(x_2) - \psi_{m+1}(x_1)\phi_n^\prime(x_2)
\\&=-\psi_{m+1}^{\prime\prime}(x_1)\psi_{n+1}(x_2) - \psi_{m+1}(x_1)\psi_{n+1}^{\prime\prime}(x_2)
\\&=-\Delta(\psi_{m+1}(x_1)\psi_{n+1}(x_2)),
\end{split}\end{equation}
which implies $\phi|_{\partial\Omega}=0.$

To measure the error, we further define the $d$-dimensional Jacobi-weighted Sobolev space below
\begin{equation*}
B^m(\Omega):=\left\{\phi: \partial_{\boldsymbol{x}}^{\boldsymbol{k}} \phi \in L^2_{\boldsymbol{\omega}^{\boldsymbol{k},\boldsymbol{k}}}(\Omega), 0 \leq|\boldsymbol{k}|_1 \leq m\right\}, \quad \forall m \in \mathbb{N}_0,
\end{equation*}
equipped with the norm and semi-norm
\begin{equation*}
\begin{aligned}
& \|\phi\|_{B^m(\Omega)}=\bigg(\sum_{0 \leq|\boldsymbol{k}|_1 \leq m}\left\|\partial_{\boldsymbol{x}}^k \phi\right\|_{\boldsymbol{\omega}^{k,k}}^2\bigg)^{1 / 2}, \;\;\; |\phi|_{B^m(\Omega)}=\bigg(\sum_{j=1}^d\left\|\partial_{x_j}^m \phi\right\|_{\boldsymbol{\omega}^{m e_j,m e_j}}\bigg)^{1 / 2}.
\end{aligned}
\end{equation*}
We consider the $H^2$-orthogonal projection $\pi^2_N:H^1_0(\Lambda^2)\cap H^2(\Lambda^2)\rightarrow (\mathcal{P}_N)^2\cap H^1_0(\Lambda^2)\cap H^2(\Lambda^2)$, defined by
\begin{equation}\label{Deltaproj}
(\Delta(\phi-\pi^{2}_N \phi),\Delta\varphi)=0,\;\;\;\varphi\in (\mathcal{P}_N)^2\cap H^1_0(\Lambda^2)\cap H^2(\Lambda^2).
\end{equation}
Then, we can derive the following approximation result of $\pi^2_N$.
\begin{lemma}
If $\phi\in H^1_0(\Lambda^2)\cap H^2(\Lambda^2)$ and $\Delta\phi\in B^r(\Lambda^2)$, there holds
\begin{equation}\label{errH2d2} 
\|\Delta(\pi_{N}^2 \phi-\phi)\|_{\Lambda^2}\leq cN^{-r}|\Delta\phi|_{B^r(\Lambda^2)}. 
\end{equation}
\end{lemma}
\begin{proof}
 In view of \eqref{Deltaproj}, we have
\begin{equation}\label{splitd}
\pi_{N}^2=\pi_{N}^{(1)}\circ \pi_{N}^{(2)},
\end{equation}
where the one dimensional orthogonal projection $\pi^{(j)}_N:H^2(\Lambda)\rightarrow\mathcal{P}_N$, $j=1,2$ such that $\pi^{(j)}_N\phi(\pm 1)=\phi(\pm1)=0$ and $(\pi^{(j)}_N\phi)^\prime(\pm 1)=\phi^\prime(\pm1)$. According to \cite[(4.82)]{ShenTangWang2011}, for any $\phi\in H^2(\Lambda)$ and $\partial^2_x\phi\in B^{r-2}_\ast(\Lambda):=\{\phi:\partial^k_x\phi\in L^2_{\omega^{k,k}}(\Lambda),0\leq k\leq r-2\}$, we have
\begin{equation}\label{errH2d1} 
\|\pi_{N}^{(j)} \phi-\phi\|_{H^\mu(\Lambda)}\leq cN^{\mu-r}(\|\phi\|_{H^2}+\|\partial_x^r\phi\|_{\omega^{r-2,r-2}}), \;\;\; \mu\in[0,2].
\end{equation}

Using integration by parts and the Cauchy-Schwarz inequality, we obtain that
\begin{equation}\label{eqx2}
\begin{split}
\|\Delta(\pi_{N}^2\phi-\phi)\|^2_{\Lambda^2}&=\|\partial^2_{x_{1}}(\pi_{N}^2 \phi-\phi)\|^2_{\Lambda^2}+\|\partial^2_{x_{2}}(\pi_{N}^{2} u-u)\|^2_{\Lambda^2}
\\&\quad+2\int_{\Lambda^2}\partial^2_{x_1}(\pi_{N}^{2} \phi-\phi)\partial^2_{x_2}(\pi_{N}^{2}\phi-\phi) {\rm d}\bs x
\\&\leq c\|\partial^2_{x_{1}}(\pi_{N}^2 \phi-\phi)\|^2_{\Lambda^2}+c\|\partial^2_{x_{2}}(\pi_{N}^{2} u-u)\|^2_{\Lambda^2}:=I_1+I_2.
\end{split}
\end{equation}
We only consider $I_1$, since the other terms in $I_2$ can be derived in a similar fashion.
By virtue of \eqref{errH2d1} with $\mu=2$ and \eqref{splitd}, we derive that 
\begin{equation}\begin{split}
\label{eqx1}
I_1&\leq 2\Big(\|\partial^2_{x_1}(\pi_{N}^{(1)} \phi-\phi)\|^2_{\Lambda^2} + \|\partial_{x_1}^2 \pi_{N}^{(1)} \circ(\pi_{N}^{(2)} \phi-\phi) \|^2_{\Lambda^2}  \nonumber
\\&\quad\quad+\|\partial^2_{x_2}(\pi_{N}^{(2)} \phi-\phi)\|^2_{\Lambda^2}+\|\partial_{x_2}^2 \pi_{N}^{(2)} \circ(\pi_{N}^{(1)}\phi-\phi) \|^2_{\Lambda^2}\Big)
\\& \leq c N^{-2r}|\Delta\phi|^2_{B^r(\Lambda^2)}+2\|\pi_{N}^{(2)} (\partial_{x_1}^2\phi)-(\partial_{x_1}^2\phi) \|^2_{\Lambda^2} +\|  \pi_{N}^{(1)} (\partial_{x_2}^2\phi)-(\partial_{x_2}^2\phi)\|^2_{\Lambda^2}
\\& \leq c N^{-2r}|\Delta\phi|^2_{B^r(\Lambda^2)}.
\end{split}\end{equation}
A combination of the above facts leads to the desired result. 
 \end{proof}

\begin{thm}\label{modelproerr}
Let $\bs u, \bs u_{N}$ be respectively the solutions of \eqref{eq: weakform} and \eqref{eq: weakformdist}. If $\bs u\in \mathbb{X}^{{\rm div}0}(\Lambda^2)$ and $\nabla\times\bs u\in B^{r}(\Lambda^2)$ with $r\ge 0$, then we have
\begin{equation}\label{modelorder}
\|\nabla\times(\bs u-\bs u_{N})\|_{\Lambda^2}+\kappa\|\bs u-\bs u_N\|_{\Lambda^2}\leq cN^{-r}|\nabla\times\bs u|_{B^r(\Lambda^2)}.
\end{equation}
where $c$ is a positive constant independent of $r$, $N$, and $\bs u$.
\end{thm}
\begin{proof}
Applying the first Strang Lemma (see, e.g., \cite{strang1972variational}) to \eqref{eq: weakformdist}, we can obtain 
\begin{equation*}
\label{AMerr1}\begin{split}
\|\nabla\times(\bs u- \bs u_{N})\|_{\Lambda^2}+\kappa\|\bs u-\bs u_N\|_{\Lambda^2}
 &\lesssim\inf_{ \bs w_{N}\in \mathbb{ X}^{{\rm div}0}_{N,2d} }\|\nabla\times(\bs u- \bs w_{N})\|_{\Lambda^2}+\kappa\|\bs u-\bs w_N\|_{\Lambda^2}. 
\end{split}\end{equation*}
Taking $\bs w_{N}= \Pi_N\bs u(\bs x)=\nabla\times\pi^2_N\phi$ in above, we obtain from \eqref{boundary2}, \eqref{errH2d2}, and Poincar\'e inequality that
\begin{equation*}\begin{split}
& \|\nabla\times(\bs u-\Pi_N\bs u)\|_{\Lambda^2}+\kappa\|\bs u-\Pi_N\bs u\|_{\Lambda^2}
\\&= \|\nabla\times\nabla\times (\phi- \pi^2_N\phi)\|_{\Lambda^2}+\kappa\|\nabla\times (\phi-\pi^2_N\phi)\|_{\Lambda^2}
\\&\leq \|\nabla\times\nabla\times( \phi-\pi^2_N \phi)\|_{\Lambda^2}
= \|\Delta(\phi-\pi^{2}_N \phi)\|_{\Lambda^2}
\\&\leq cN^{-r}|\Delta\phi|_{B^r(\Lambda^2)}=cN^{-r}|\nabla\times\bs u|_{B^r(\Lambda^2)}. 
\end{split} \end{equation*} 
This ends the proof.
\end{proof}

\subsection{Solution algorithm based on fast diagonalizable auxiliary problem}\label{sec2.3}
In this subsection, we propose a highly efficient solution algorithm for the divergence-free spectral approximation scheme \eqref{eq: weakformdist}, which draws inspirations from the fast diagonalisation method (cf. \cite{buzbee1970direct, haidvogel1979accurate,Shen94b}) for Poisson-type equations and the matrix-free preconditioned Krylov subspace iterative method.

\begin{prop}
Denote symmetric matrices $M=(M_{mn})$, $I=(I_{mn})$ and $S=(S_{mn})$, where ${1\leq m,n\leq N-1}$ and the entries are given by  


\begin{align}
& M_{mn}=\big( \psi_{n+1},\psi_{m+1}\big)=
\begin{cases}
 \dfrac{1}{2n+1}\big( \dfrac{1}{2n-1}+\dfrac{1}{2n+3}  \big),\quad & m=n,\\[5pt]
 - \dfrac{1}{\sqrt{2n+1}}\dfrac{1}{\sqrt{2n+5}}\dfrac{1}{2n+3}     , \quad & m=n+2,\\[5pt]
0,\quad & {\rm otherwise};
\end{cases} \label{eq: M}\\
&      I_{mn}=\big(\phi_n,\phi_m \big )=\delta_{mn}; \label{eq: I}  \\
&S_{mn}=\big(\phi_n',\phi_m'\big)=\frac{1}{4}{\sqrt{(2m+1)(2n+1)}}\, n(n+1)\big(1+(-1)^{m+n}  \big),\;\; m\geq n. \label{eq: S}
\end{align}
\end{prop}
\begin{proof}
The analytic expressions of $M_{mn}$ and $I_{mn}$ in equations \eqref{eq: M}-\eqref{eq: I} have been given in \cite{Shen94b} and can be easily derived with the help of the orthogonality of Legendre polynomials in equation \eqref{eq: Linner} and the recurrence relation
\begin{equation}\label{eq: reccu}
(2n+1)L_n(\xi)=L_{n+1}'(\xi)-L_{n-1}'(\xi).
\end{equation} 
Next, we look into the derivation of $S$. Since $S$ is symmetric, we consider only $m\geq n$. By integration by parts, one obtains 
\begin{equation}\label{eq: Smiddle}
S_{mn}=\frac{\sqrt{(2m+1)(2n+1)}}{2} \Big\{ [L_n'(\xi)L_m(\xi)]\big|_{-1}^1-\int_{-1}^1 L_n''(\xi)L_m(\xi)\,{\rm d}\xi  \Big\}.
\end{equation}
By the fact that when $m\geq n$, $L_n''(\xi) \in \mathcal{P}_{m-2}$, thus $L_n''(\xi)$ can be expanded by a linear combination of $\{L_k(\xi) \}_{k=0}^{m-2}$. Thus, the last term in the above equation vanishes due to orthogonality \eqref{eq: Linner}. Note that the boundary values of Legendre polynomials satisfy
\begin{equation}\label{eq: bcL}
L_n(\pm 1)=(\pm 1)^n, \quad L_n'(\pm 1)=\frac{1}{2}(\pm 1)^{n-1}n(n+1).
\end{equation}
Inserting equation \eqref{eq: bcL} into equation \eqref{eq: Smiddle} directly leads to the analytic expression \eqref{eq: S}.
\end{proof}

One can expand the numerical solution $\bs u_N(\bs x)$ of the approximation scheme \eqref{eq: weakformdist} by the divergence-free basis functions $\bs u_N(\bs x)=\sum_{m,n=1}^{N-1} u_{mn} \bs \Phi_{mn}(\bs x),$
and denote
\begin{equation*}\label{eq: umatrix}
f_{mn}=\big(\bs f,\bs \Phi_{mn}\big),\quad F=(f_{mn})_{m,n=1,\cdots, N-1},\quad U=(u_{mn})_{m,n=1,\cdots, N-1}.
\end{equation*}
Taking $\bs v(\bs x)=\bs \Phi_{\hat m \hat n}(\bs x)$ in equation \eqref{eq: weakform} for $\hat m,\hat n=1,\cdots, N-1$,  and noticing the definitions of matrices $M$, $S$ and $I$ in equations \eqref{eq: M}-\eqref{eq: S}, one arrives at the following matrix equation 
\begin{equation}\label{eq: meq2d}
SUM+2 IUI + MUS+\kappa \big( MUI+IU M  \big)=F.
\end{equation}

One can remap the unknown matrix $U$ and source term matrix $F$ into vector forms ${\rm vec}(U)$ and ${\rm vec}( F)$, following a natural index ordering as follows:
\begin{equation*}\label{eq: 2Dlong}
\{ ({\rm vec}(U))_p, ({\rm vec}( F) )_p \}_{p=1,\cdots,(N-1)^2}=\{ u_{mn}, f_{mn} \}_{m,n=1}^{N-1}, \;\; p=m+(N-1)(n-1).
\end{equation*}
The matrix equation \eqref{eq: meq2d} can be written equivalently into a linear system $A\,{\rm vec}(U)$ $={\rm vec}( F)$, with the global matrix $A$ be assembled by 
\begin{equation}\label{eq: Aform}
A:=M \otimes S+ S \otimes M+2 I\otimes I +\kappa (I \otimes M+ M \otimes I).
\end{equation}

In order to design fast solution algorithm, it is necessary to look into the sparsity of the global matrix $A$. 
Though $M$ is penta-diagonal, due to the dense nature of the stiffness matrix $S$, the resultant global matrix $A$ is a dense matrix. Thus, direct solvers based on fast factorization \cite{golub2013matrix, trefethen2022numerical} that perform well for sparse matrix will no longer work.  Generally speaking, matrix-free Krylov subspace iteration methods \cite{kelley1995iterative, saad2003iterative} would be the method of choice for solving such large dense linear systems formed by summation of Kronecker products of sub-matrices, due to significant reduction on the number of floating point operations needed for operator evaluation and memory comsumption,  compared with solving the fully assembled system. Nevertheless, these iterative methods would encounter severe convergence issue for ill-conditioned system, arising from high-order spectral discretisation. 

On the other hand, it is known that matrix diagonalisation method \cite{buzbee1970direct, haidvogel1979accurate,Shen94b} is an efficient solution algorithm for  solving Poisson-type equations. The essential idea of this technique resides in that a symmetric matrix and an identity matrix can be simultaneously diagonalised by the same orthonormal matrix. However, it is highly non-trivial and seems impoissble to extend this idea to the current problem as it involves the simultaneous diagonalisation of multiple matrices $M$, $S$ and $I$. In what follows, we show how to link these matrices through two auxiliary eigenvalue problems and reveal the fact that these three matrices are almost simultaneously diagonalizable.

%

 To facilitate the development of the efficient algorithm for solving matrix equation \eqref{eq: meq2d}, one needs to resort to the following two auxiliary eigenvalue problems:
\begin{equation}\label{eq: biha}
y^{(4)}(x)=\lambda y(x), \;\; x\in \Lambda,\quad y(\pm1)=0,\;\; y''(\pm 1)=0,
\end{equation}
and
\begin{equation}\label{eq: Lap}
-w^{(2)}(x)=\mu w(x), \;\; x\in \Lambda,\quad w(\pm1)=0.
\end{equation}
Their weak formulations are as follows: 
\begin{itemize}
\item the biharmonic eigenvalue problem: find $y(x)\in \mathring{H}^2(\Lambda):=\big \{u\in H^2(\Lambda),$ $u(\pm 1)=0\big\}$ and $\lambda\in \mathbb{R}$, such that
\begin{equation}\label{eq: wkbih}
(y'', v'')_{L^2(\Lambda)}=\lambda(y,v)_{L^2(\Lambda)},\;\; \forall \,v(x)\in \mathring{H}^2(\Lambda), 
\end{equation}
\item the Laplace eigenvalue problem: find $w(x)\in H^1_0(\Lambda)$ and $\mu \in \mathbb{R}$, such that
\begin{equation}\label{eq: wkLap}
(w', v')_{L^2(\Lambda)}=\mu(w,v)_{L^2(\Lambda)},\;\; \forall \, v(x)\in H^1_0(\Lambda).
\end{equation}
\end{itemize}
Let us define the approximation space $\mathbb{Y}_N:={\rm span}\{ \psi_{m+1}(x) \}_{m=1}^{N-1}$ for equations \eqref{eq: wkbih} and \eqref{eq: wkLap}, the related approximation schemes for these two problems are to 
\begin{itemize}
\item find $y_{N}\in \mathbb{Y}_N$ and $\lambda_{h} \in \mathbb{R}$, such that
\begin{equation}\label{eq: wkbih1}
(y_{N}'', v'')_{L^2(\Lambda)}=\lambda_{h}(y_{N},v)_{L^2(\Lambda)},\;\; \forall v(x)\in \mathbb{Y}_N, 
\end{equation}
\item and find $w_{N}\in \mathbb{Y}_N$ and $\mu_{h} \in \mathbb{R}$, such that
\begin{equation}\label{eq: wkLap1}
(w_{N}', v')_{L^2(\Lambda)}=\mu_{h}(w_{N},v)_{L^2(\Lambda)},\;\; \forall v(x)\in \mathbb{Y}_N.
\end{equation}
\end{itemize}
Expanding $y_{N}(x)$ and $w_{N}(x)$ by
\begin{equation}\label{eq: YN}
y_{N}(x)=\sum_{n=1}^{N-1} y_n \psi_{n+1}(x),\qquad w_{N}(x)=\sum_{n=1}^{N-1} w_n \psi_{n+1}(x),
\end{equation}
and taking $v(x)=\psi_{m+1}(x)$ for $m=1,\cdots, N-1$ in \eqref{eq: wkbih1} and \eqref{eq: wkLap1}, respectively, one obtains
\begin{align}
& \sum_{n=1}^{N-1}\big( \psi_{n+1}'', \psi_{m+1}'' \big)_{L^2(\Lambda)}y_{n,i}=\lambda_{h,i} \sum_{n=1}^{N-1} \big(\psi_{n+1}, \psi_{m+1} \big)_{L^2(\Lambda)}y_{n, i}, \label{eq: YN1} \\
& \sum_{n=1}^{N-1}\big( \psi_{n+1}', \psi_{m+1}' \big)_{L^2(\Lambda)}w_{n,i}=\mu_{h,i} \sum_{n=1}^{N-1} \big(\psi_{n+1}, \psi_{m+1}\big)_{L^2(\Lambda)}w_{n, i},\label{eq: WN1}
\end{align}
where $\lambda_{h,i}$ and $\mu_{h,i}$ are the $i$-th  eigenvalues of the generalized matrix eigenvalue problems \eqref{eq: YN1} and \eqref{eq: WN1}, respectively, and $y_{n,i}$ and $w_{n,i}$ are respectively the $n$-th expansion coefficients of the $i$-th eigenvector under the basis $\{ \psi_{n+1}(x) \}_{n=1}^{N-1}$ of equations \eqref{eq: YN1} and \eqref{eq: WN1}. Aware of the derivative property \eqref{eq:psirelaphi} and the definition of matrix $M$,$S$, and $I$ in equations \eqref{eq: M}-\eqref{eq: S}, the above equations \eqref{eq: YN1}-\eqref{eq: WN1} lead to  the following generalized matrix eigenvalue problems as follows:
\begin{equation}\label{eq: meigen2}
I E =M E  \Theta,\quad SE_1=ME_1 \Theta_1.
\end{equation}
Here, $E,$ $E_1$ are orthonormal matrices and $\Theta$ and $\Theta_1$ are diagonal matrices of size $(N-1)$-by-$(N-1)$, respectively, defined by
$(E)_{ni}=w_{n,i},$ $(E_1)_{ni}=y_{n,i}$ and 
\begin{equation}\label{eq: ELamb}
\Theta={\rm diag}(\mu_{h,1},\cdots, \mu_{h,N-1}),\quad  \Theta_1={\rm diag}(\lambda_{h,1},\cdots,\lambda_{h,N-1}).
\end{equation}

It is crucial to observe that the eigenvalues and eigenvectors of equations \eqref{eq: biha} and \eqref{eq: Lap} have analytic expressions and are closely related to each other as given below: for $i=1,\cdots, N-1$,
\begin{equation}
\label{eq: analeigen}
y_i(x)=w_i(x)=\sin\big( { i}{\pi}(x+1)/{2} \big), \;\;  \mu_i=\big({ i}{\pi}/{2} \big)^2,  \;\; \lambda_i=\mu_i^2.
\end{equation} 
Thus, one can infer readily that $E_1 \approx E$ and $\Theta_1\approx \Theta^2$, and as a result, $MS\approx SM \approx I$.  Note that if $M$ and $S$ commute, i.e. $MS=SM$, there exists a set of common orthonormal eigenvectors such that $M$ and $S$ are simultaneously diagonalizable (see e.g. \cite{lax2007linear} p.74, Theorem 14). Next, we show how close $MS$ and $SM$  to the identity matrix in the lemma below.
\begin{lm}\label{prop: MS} $MS$ and $SM$ are mutually transposed matrices of size $(N-1)$-by-$(N-1)$, with the following analytic expression:
\begin{equation} \label{eq: MS}\begin{split}
&(SM)_{nm}=(MS)_{mn}\\&
=\begin{cases}
\delta_{mn},\quad & 1\leq m\leq N-3,\\[5pt]
 \delta_{(N-2)n}+\dfrac{1}{4}\sqrt{\dfrac{2n+1}{2N-3} } \dfrac{n(n+1)}{2N-1} \Big( 1+(-1)^{n+N}  \Big)   , \;\; & m=N-2,\\[5pt]
 \delta_{(N-1)n}+\dfrac{1}{4}\sqrt{\dfrac{2n+1}{2N-1} } \dfrac{n(n+1)}{2N+1} \Big( 1+(-1)^{n+N+1}  \Big),\;\; & m=N-1.
\end{cases}
\end{split}\end{equation}
\end{lm} 
\begin{proof}
We postpone the detailed proof in Appendix \ref{AppendixA0}.
\end{proof}

It can be observed from Lemma \ref{prop: MS} that $MS$ is a lower triangular matrix, with its first  $(N-3)$-by-$ (N-3)$ sub-matrix been identity. This inspires us to approximate $S$ by $M^{-1}$.  Note that from equation \eqref{eq: meigen2}, one has
\begin{equation}\label{eq: Mdiag}
M=E \Theta^{-1} E^{\intercal}=E D E^{\intercal},
\end{equation}
where $D={\rm diag}(d_1,\cdots,d_{N-1}):=\Theta^{-1}$ is a diagonal matrix with its main diagonal entries the eigenvalues of $M$ and $E$ is the corresponding orthonormal eigenvector matrix. Correspondingly, $S$ is approximated by
\begin{equation}\label{eq: Sapprox}
S\approx M^{-1}=E D^{-1} E^{\intercal}. 
\end{equation}
\vspace{-16pt}\begin{rem}\label{rem: Sapprox} {\em 
The numerical eigenvalues/eigenvectors of $M$ and $S$ in equations \eqref{eq: meigen2}, i.e. $E, E_1$ and $\Theta, \Theta_1$, respectively, could also be approximated by their analytic counterparts. To be more specific, $y_{n,i}$, $w_{n,i}$ and $\lambda_{h,i}$ and $\mu_{h,i}$ can be approximated by $\lambda_{h,i}\approx \lambda_i= \big(\frac{\pi}{2}i \big)^4,$ $\mu_{h,i}\approx \mu_i= \big(\frac{\pi}{2}i \big)^2, 
$ $y_{n,i}\approx w_{n,i}\approx \hat y_{n,i}=\hat w_{n,i},$ and
\begin{equation}\label{eq:approxy}
\sum_{n=1}^{N-1}  \hat y_{n,i} \psi_{n+1}(x)=\sum_{n=1}^{N-1} \hat  w_{n,i} \psi_{n+1}(x)\approx \sin\Big(\frac{\pi}{2} i(x+1) \Big),\;\; i=1,\cdots N-1. 
\end{equation}
By taking derivative on both sides of the above equation, using equation \eqref{eq:psirelaphi} and the orthogonality of Legendre polynomials under the $L^2$ inner product, one readily verifies
\begin{equation}\label{eq: ynj}
\hat y_{n,j}=\hat w_{n,j}=\pi \sqrt{\frac{(2n+1)j}{2}}  \cos \Big(\frac{\pi}{2}(j+n)   \Big)  J_{n+\frac{1}{2}}\Big(\frac{\pi}{2}j \Big).
\end{equation}
Here, $J_n(r)$ is the Bessel function of the first kind and the following identity has been used
\begin{equation}\label{eq: sbessel}
\int_{-1}^1 e^{\ri r x}L_n(x)dx=\ri^n \sqrt{\frac{2\pi}{r}}J_{n+1/2}(r),
\end{equation}
where $\ri$ is the complex unit. 
Thus, define $\hat D={\rm diag}(\hat d_1,\cdots, \hat d_{N-1}),$ $\hat d_j=\big(\frac{\pi}{2}i\big)^{-2}$, and 
\begin{equation}\label{eq: hatE}
(\hat E)_{ij}=\pi \sqrt{\frac{(2i+1)j}{2}}  \cos \Big(\frac{\pi}{2}(j+i)   \Big)  J_{i+\frac{1}{2}}\Big(\frac{\pi}{2}j \Big). 
\end{equation}
$M$ and $S$ could also be approximated by
\begin{equation}\label{eq: MSanaAp}
M \approx \hat E \hat D \hat E^{\intercal}, \quad S \approx \hat E \hat D^{-1} \hat E^{\intercal}.
\end{equation}
However, further numerical evidence illustrates the deficiency of using the above approximation, compared with using the approximation in equation \eqref{eq: Sapprox}.}
\end{rem}

 \begin{figure}[tbp]
\begin{center}
  \subfigure[Eigenvalues of $S$: exact vs approximation ]{ \includegraphics[scale=.43]{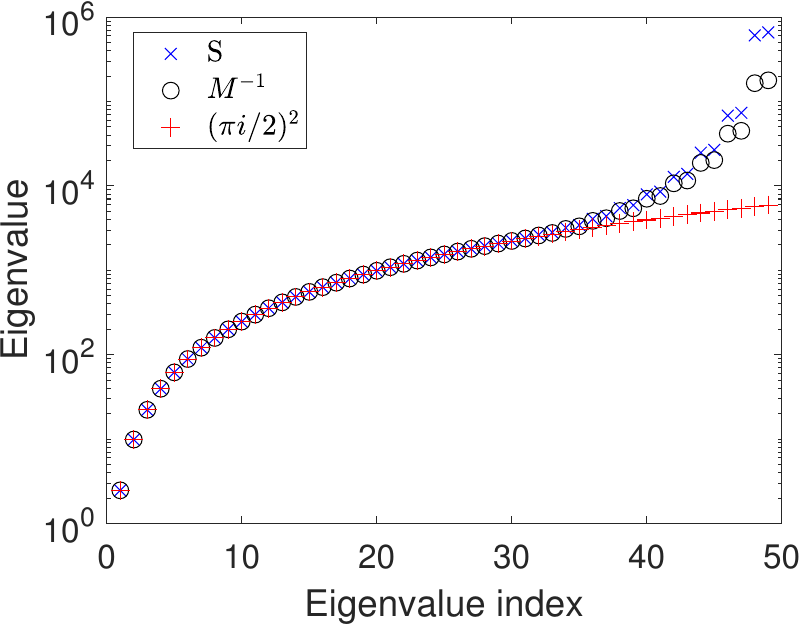}}\qquad
  \subfigure[Eigenvectors of $S$: exact vs approximation  ]{ \includegraphics[scale=.43]{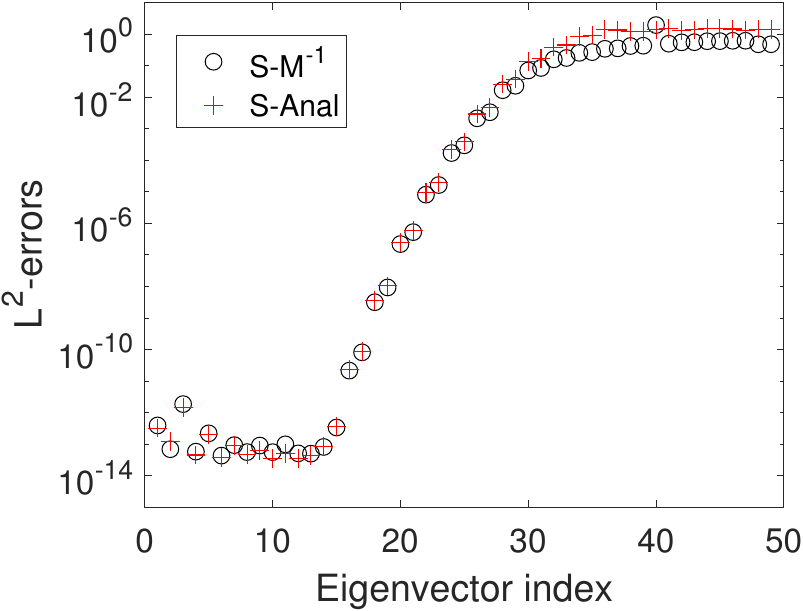}}
    \caption{\small Comparisons of the exact eigenvalues and eigenvectors of $S$ with their approximations. 
(a) Eigenvalues of $S$ (in ascending order) as a function of the eigenvalue index $i$ and compare them with the eigenvalues of $M^{-1}$ and the analytic expression ${(\pi i/2)^2}$; (b) the differences in $l^2$-norm between the associated eigenvectors of $S$ and  $E$ (the eigenvector matrix of $M^{-1}$) and the analytic expression $\hat E$ in equation \eqref{eq: hatE}.   } 
   \label{figs: Sapprox}
\end{center}
\end{figure}

In order to quantify the accuracy of the approximations \eqref{eq: Sapprox} and \eqref{eq: MSanaAp}, let us denote the eigenvalues of S by $\{ s_i \}$ (in ascending order) and the corresponding orthonormal eigenvector matrix by $E^S$.
In Fig.\,\ref{figs: Sapprox} (a), we fix the polynomial order $N=50$ and depict the spectrum of $S$ and $M^{-1}$, as a function of the eigenvalue index $i$ in ascending order. The analytic formula $(\pi i/2)^2$ obtained from equation \eqref{eq: analeigen} is also included for comparison. One can observe that the differences of the top $70\%$ of these eigenvalue pairs are negligible.  While for the remaining $30\%$ of the eigenvalues, the result obtained from $M^{-1}$ is obviously better than that from the analytic formula  $(\pi i/2)^2$. In Fig.\,\ref{figs: Sapprox} (b), the difference in $l^2$-norm of the eigenvectors between $S$ and $M^{-1}$ and the analytic formula in \eqref{eq: MSanaAp}, i.e., $\|E^S_{:,i}-E_{:,i}\|_2$ and  $\|E^S_{:,i}-\hat E_{:,i}\|_2$ are depicted as a function of the eigenvector index $i$. It can be seen that the accuracy of both approximations to the top $30\%$ of the eigenvectors is up to the machine roundoff error.  While for the rest of the $70\%$ of the eigenvectors, the approximation errors gradually increase from $10^{-14}$ to $10^{-1}$ for both approximations. Moreover, it can be observed clearly that in general,  $\|E^S_{:,i}-E_{:,i}\|_2$ is smaller than $\|E^S_{:,i}-\hat E_{:,i}\|_2$.  These numerical evidences demonstrate that $M^{-1}$ serves as a viable approximation of $S$, especially for the low and medium frequency modes.  This observation plays a significant role in  the development of the proposed method. 

In what follows, we construct an auxiliary matrix equation by replacing $S$ by $M^{-1}$ in equation \eqref{eq: meq2d} and show that this matrix equation is fully diagonalizable and easy to compute. 
\begin{prop}\label{prop: aux2d}
Define an auxiliary matrix equation for the original system \eqref{eq: meq2d} as
\begin{equation}\label{eq: meq2daux}
M^{-1}UM+2 IUI + MUM^{-1}+\kappa \big( MUI+IU M  \big)=F,
\end{equation}
which can be written equivalently as the following linear system
\begin{equation}\label{eq: tAform}
\tilde A\,{\rm vec}( U)={\rm vec}( F),
\end{equation}
with $\tilde A:=M \otimes M^{-1}+ M^{-1} \otimes M+2 I\otimes I +\kappa (I \otimes M+ M \otimes I).$ 
Here, $M$ is the penta-diagonal matrix defined in equation \eqref{eq: M} and can be diagonalized by \eqref{eq: Mdiag}. The solution $U$ takes the form 
\begin{equation}\label{eq:U2dsolu}
U=EVE^{\intercal},\quad \text{with}\,\,\,\,   V_{ij}=\frac{(E^{\intercal} F E)_{ij}}{(d_i/d_j+d_j/d_i+2)+\kappa (d_i+d_j) }.  
\end{equation}
\end{prop}
\begin{proof}
Inserting equation \eqref{eq: Mdiag} into equation \eqref{eq: meq2daux} and introducing the auxiliary matrix $V=E^{\intercal} U E$, one readily obtains
\begin{equation*}
D^{-1} V D+DVD^{-1}+2V+\kappa (DV+VD)=E^{\intercal} F E,
\end{equation*}
which directly leads to the desired formula \eqref{eq:U2dsolu}.
\end{proof}

A comprehensive numerical investigation of the spectrum of ${\tilde A}^{-1} A$ leads to the following observations.
\begin{enumerate}
\item[(i)] The eigenvalues of $A$ spread widely in between $-10^4$ to $10^7$, while that of ${\tilde A}^{-1} A$ is confined in between $1$ to $40$ and more than $96\%$ of the eigenvalues of $\tilde A^{-1} A$ are clustered at $1$, computed with a fixed $N=120$ and various $\kappa=-2000,-5000,-10000$, as can be seen in Fig.\,\ref{figs: preeigdis}.
\item[(ii)] The differences in $l^{\infty}$-norm between the first $13689$ eigenvalues (i.e. $(120-3)^2$) and 1 is negligible, and the differences jump abruptly to $O(1)$ for the rest of the eigenvalues,  indicating that there is an extremely large invariant subspace associated with the dominate eigenvalue $1$, as shown in Fig.\,\ref{figs: preeigdis} (b). 
\end{enumerate}
 \begin{figure}[tbp]
\begin{center}
  \subfigure[Spectrum of $A$]{ \includegraphics[scale=.3]{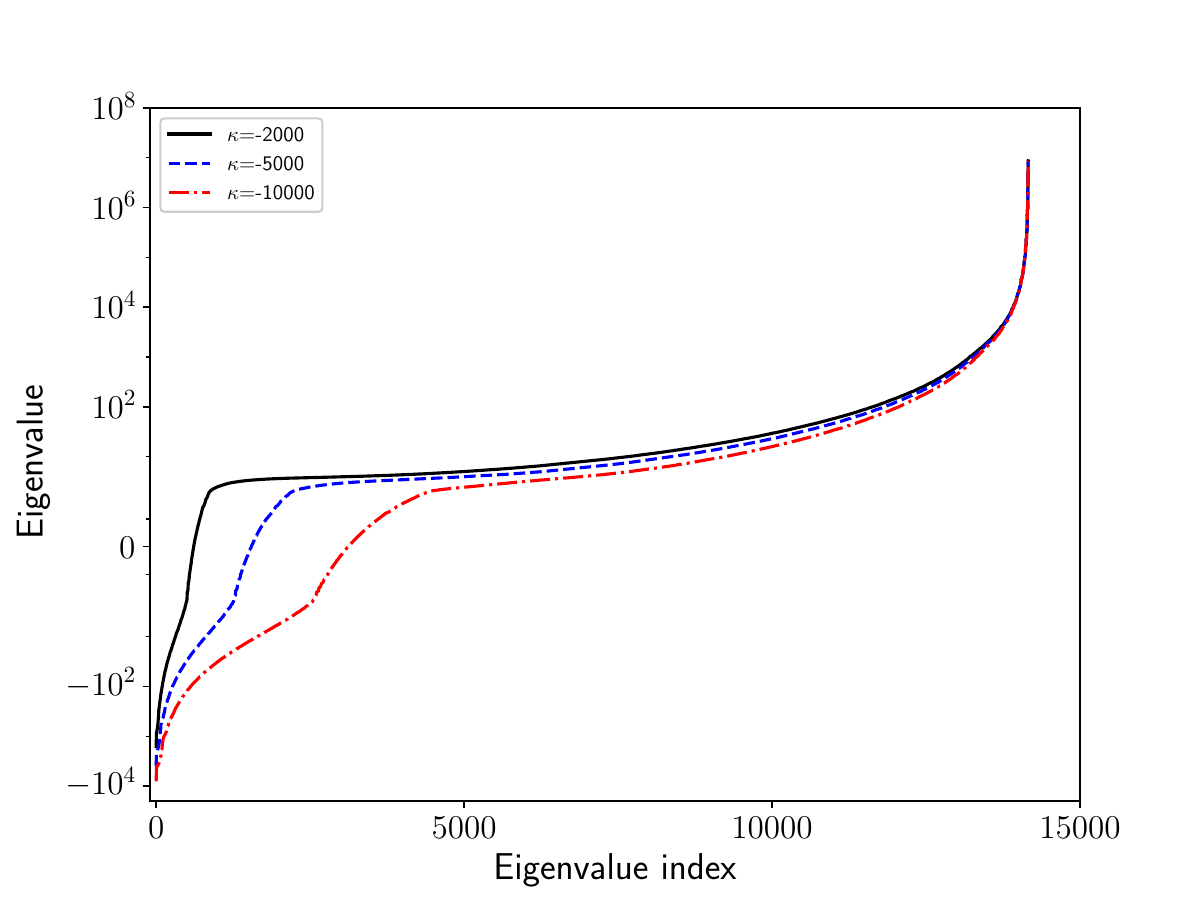}}\hspace{-8pt}
  \subfigure[Spectrum of $\tilde A^{-1} A$]{ \includegraphics[scale=.39]{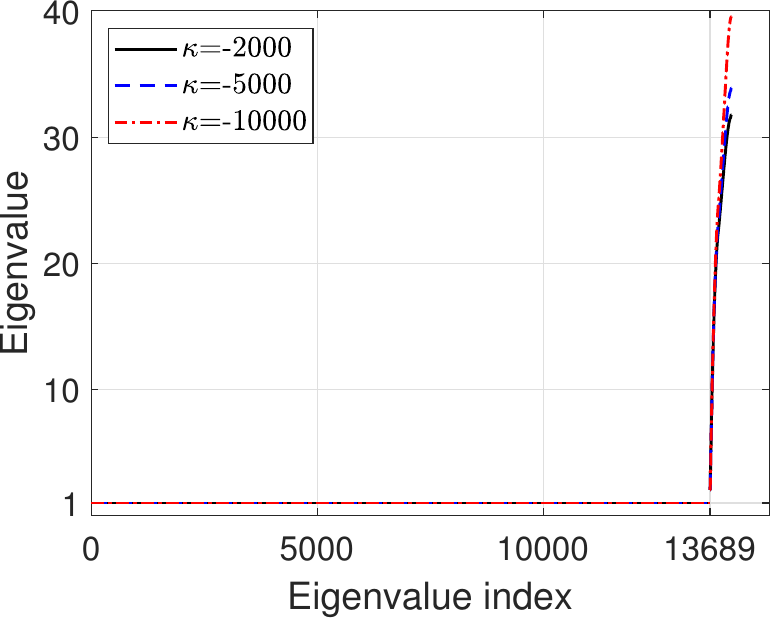}}
     \caption{\small Eigenvalue distributions of A and $\tilde A^{-1}A$ in 2D. (a) Spectrum of $A$; (b) spectrum of $\tilde A^{-1}A$; (c) the differences in $l^{\infty}$-norm between the eigenvalues of $\tilde A^{-1}A$ and $1$. These results are computed for various $\kappa=-2000,-5000,-10000$ and a fixed polynomial order $N=120$.} 
   \label{figs: preeigdis}
\end{center}
\end{figure}

\begin{thm}\label{thm:2dinvarspace}
The dimension of the invariant subspace of the $(N-1)^2$-by-$(N-1)^2$ matrix $\tilde A^{-1} A$  with respect to the eigenvalue $1$ is $(N-3)^2$.
\end{thm}
\begin{proof}
It suffices to find $(N-3)^2$ linearly independent solutions ${\rm vec}(U)$ that satisfy
\begin{equation*}
A\, {\rm vec} (U)=\tilde A\, {\rm vec} (U),
\end{equation*}
or more specificically, $(N-3)^2$ linearly independent  $U$ that satisfy
\begin{equation*}
SUM+ MUS=M^{-1}UM + MUM^{-1}. 
\end{equation*}
Let us introduce an auxiliary matrix $\bar V$ such that $U=M\bar VM$, the above equation can be written equivalently as
\begin{equation}\label{thm: prove3}
(SM-I)\bar VM^2+M^2 \bar V (MS-I)=0.
\end{equation}
With the help of the analytic expressions of $MS$ and $SM$ in Lemma \ref{prop: MS}, one can obtain that all the elements in the first $(N-3)$ columns of $SM-I$ are zeros. Thus, by choosing $\bar V$ be a $(N-1)$-by-$(N-1)$ matrix defined by
\begin{equation*}
\bar V=
\begin{cases}
\bar v_{ij},& i,j \leq N-3,\\
0, & {\rm otherwise}, 
\end{cases}
\end{equation*}
where $\{ \bar v_{ij} \}_{i,j=1}^{N-3}$ are $(N-3)^2$ degree of freedom, one can readily verify that $\bar V$ satisfies equation \eqref{thm: prove3}. This ends the proof. 
\end{proof}

These observations and Theorem \ref{thm:2dinvarspace} inspire us to combine the auxiliary problem in Proposition \ref{prop: aux2d} with a preconditioned Krylov subspace iterative method. In what follows, we present a highly efficient solution algorithm, based on a proper combination of the fast diagonalisation solver for the auxiliary problem, the matrix-free operator evaluation approach without the assembly of global matrix, and a preconditioned GMRES method. The algorithm is summarised  in Algorithm \ref{ag:ag1}. 

\begin{algorithm}
\caption{Fast divergence-free spectral algorithm for equation \eqref{eq:system}}
\label{ag:ag1}
\centering
\begin{algorithmic}[1] 
\REQUIRE {The polynomial order $N$, the matrices $M$ and $S$ in equation \eqref{eq: M}-\eqref{eq: S}, the matrices $E$ and $D$ in \eqref{eq: Mdiag}, the basis $ \big \{\bs\Phi_{m,n} \big \}_{m,n=1}^{N-1}$ in equation \eqref{eq: div0b2d}, the positive integer $J_{\rm max}$ and the threshold $\varepsilon$}
\vspace{1pt}
\ENSURE{The divergence-free expansion coefficients $U$ of the numerical solution $\bs u_N$}
\vspace{1pt}
\STATE {Compute $F=(f_{mn})_{m,n=1,\cdots, N-1},$ with $f_{mn}=\big(\bs f,\bs \Phi_{mn}\big)$}
\vspace{1pt}
\STATE Evaluate $U_0=(u^0_{pq})_{p,q=1,\cdots, N-1}$ with $U_0=E^{\intercal} F E$
\vspace{1pt}
\STATE Assign $u^0_{pq}\leftarrow u^0_{pq}/\big\{ (d_p/d_q+d_p/d_q+2)+\kappa (d_p+d_q)  \big\} $ 
\vspace{1pt}
\STATE Evaluate $R=(r_{mn})_{m,n=1,\cdots, N-1}:=F-\big(SU_0M+2U_0+ MU_0S+\kappa( MU_0+U_0 M )\big)$

\STATE Assign $R \leftarrow E\tilde R E^{\intercal}$, $\bs r \leftarrow {\rm vec}(R)$ and $j\leftarrow 0$, compute $\delta=\| \bs r \|_2$, $\rho=\delta$,  $\gamma=\|{\rm vec}(F) \|_2$ and $\bs v^{(1)}=\bs r/\delta$
 \WHILE{  $\rho/\gamma>\varepsilon$, and  $j\leq J_{\rm max}$}
\STATE Assign $j\leftarrow j+1$
\STATE Reshape $\bs v^{(j)}$ to a matrix $V$ using natural column indexing and compute 
\vspace{-6pt}
$$W=(w_{mn})_{m,n=1,\cdots, N-1}=SV M+2 V + MVS+\kappa ( MV+V M  )$$
\vspace{-14pt}
\STATE Evaluate $\tilde W=(\tilde w_{pq})_{p,q=1,\cdots, N-1}$ with $\tilde W=E^{\intercal} W E$
\vspace{1pt}
\STATE Assign $\tilde w_{pq}\leftarrow \tilde w_{pq}/\big\{ (d_p/d_q+d_q/d_p+2)+\kappa (d_p+d_q)  \big\} $
\vspace{1pt}
\STATE Assign $W=E\tilde W E^{\intercal}$ and reshape the matrix $W$ to a vector $ \bs w$ using natural column indexing
\FOR{$i=1,\cdots,  j$}
\STATE Evaluate $h_{i,j}= (\bs w, \bs v^{(i)})$, assign $\bs w\leftarrow \bs w-h_{i,j}\bs v^{(i)}$
\ENDFOR
\STATE Evaluate $h_{j+1,j}=\|\bs w \|_2$ and $\bs v^{(j+1)}=\bs w /h_{j+1,j}$
\STATE Define  $H_j=\{ h_{p,q} \}_{1\leq p \leq j+1; 1\leq q\leq j}$ and $\bs e_1=(1,0,\cdots,0)^{\intercal}\in \mathbb{R}^{j+1}$
\STATE Evaluate ${ \bs y^{(j)}={\rm argmin}_{\bs y}  \|\delta \bs e_1-H_j \bs y \|_2 }$, assign $ \rho= \|\delta \bs e_1-H_j \bs y^{(j)} \|_2 $
\ENDWHILE
\STATE Define $V_j:=[\bs v^{(1)}, \cdots, \bs v^{(j)} ]$, evaluate $\bs x=  V_j \bs y^{(j)}$
\STATE Reshape $\bs x$ to matrix $X$ and compute $U=U_0+X$
\end{algorithmic}
\end{algorithm}

We expect that $\tilde A^{-1}$ be not only an effective preconditioner for $A$, but also the number of iterations will decrease, as the polynomial order $N$ increases, since the portion of the invariant subspace of $\tilde A^{-1} A$ w.r.t. the eigenvalue 1 increases as $({N-3})^2/({N-1})^2$.

Moreover, the proposed method can be readily extended  to solve the curl-curl problem with variable coefficients 
\begin{equation}\label{eq:systemvar2d}
\nabla\times \big(\alpha(\bs x)\nabla\times (\beta(\bs x) {\bs u}(\bs x) ) \big )+\kappa \, {\bs u}(\bs x)={\bs f}(\bs x), \;\;\; \; \bs x\in \Lambda^2,
\end{equation}
where $\alpha,\beta$ are functions depending on $\bs x$  and $\nabla \cdot \bs f=0$. This can be accomplished with minor modifications in Algorithm \ref{ag:ag1} as follows:
\begin{enumerate}
\item Replace step 4 in Algorithm \ref{ag:ag1} by evaluating
\begin{equation}\label{eq: mfreeRvar}\begin{split}
R&=(r_{mn})_{m,n=1,\cdots, N-1}\\&:=\big(\bs f,\bs \Phi_{mn}\big)-\big( \alpha(\bs x) \nabla\times (\beta(\bs x) \bs u_0),\nabla \bs \Phi_{mn}\big)+\kappa (\bs u_0,\bs \Phi_{mn}),
\end{split}\end{equation}
and using Gaussian quadrature, where 
$
{\bs u}_0(\bs x)=\sum_{m,n=1}^{N-1} u^0_{mn} \bs \Phi_{mn}(\bs x).
$
\item Replace step 8 in Algorithm \ref{ag:ag1} by reshaping $\bs v^{(j)}$ to a matrix $V$ using natural column indexing and computing 
\begin{equation}\label{eq: Wvar}
W=(w_{mn})_{m,n=1,\cdots, N-1}=\big( \alpha(\bs x) \nabla\times (\beta(\bs x) \bs v) ,\nabla \bs \Phi_{mn}\big)+\kappa (\bs v,\bs \Phi_{mn} ), 
\end{equation}
where  $\bs v(\bs x)=\sum_{m,n=1}^{N-1} (V)_{mn} \bs \Phi_{mn}(\bs x)$.
\end{enumerate}
\begin{rem}{\em 
Note that one can split the penta-diagonal matrix $M$ into two symmetric tridiagonal sub-matrices and compute its associated eigenvalues and eigenvectors in $O(N^2)$ operations (see, e.g. a detailed description of the solution procedure in \cite{zhang2017optimal}). The overall computational cost is dominated by the  matrix-matrix multiplication in step 2 and 9 in Algorithm \ref{ag:ag1}, which is only a small multiple of $N^3$ operations, as opposed to $O(N^4)$ operations for assembled global system.}
\end{rem}

\section{A fast algorithm for the three-dimensional curl-curl equation}
The proposed divergence-free spectral method and its fast solution algorithm in two dimensions could be straightforwardly extended to the solution of three-dimensional curl-curl problem with both constant and variable coefficients. Let us propose the divergence-free approximation space first.

\begin{prop}
$\mathbb{ X}_{N,3d}^{{\rm div}0}(\Lambda^3)$ is a divergence-free approximation space for $\mathbb{ X}^{{\rm div}0}(\Lambda^3)$ taking the form
\begin{equation*}
\mathbb{ X}_{N,3d}^{{\rm div}0}= {\rm span}\Big( \Big \{  \bs\Phi_{m,n,l}^1,\bs\Phi_{m,n,l}^2  \Big\}_{m,n,l=1}^{N-1} ,\, \Big\{\bs\Phi^{x}_{n,l}\Big\}_{n,l=1}^{N-1},\, \Big\{\bs\Phi^{y}_{m,l} \Big\}_{m,l=1}^{N-1},\, \Big\{\bs\Phi^{z}_{m,n} \Big\}_{m,n=1}^{N-1} \Big).
\end{equation*}
Here, $\{ \bs \Phi_{m,n,l}^1(\bs x),\, \bs \Phi_{m,n,l}^2(\bs x) \}$ stands for the interior modal basis functions defined by
\begin{equation}\label{eq:divBb1}
\begin{aligned}
&\bs \Phi_{m,n,l}^1(\bs x)=\Big( \psi_{m+1}(x_1)\phi_{n}(x_2)\phi_{l}(x_3), \;\; -\phi_m(x_1)\psi_{n+1}(x_2)\phi_{l}(x_3),\;\;0\Big)^{\intercal},\\
&\bs \Phi_{m,n,l}^2(\bs x)=\Big( \psi_{m+1}(x_1)\phi_{n}(x_2)\phi_{l}(x_3),\;\; 0,\;\;   -\phi_{m}(x_1)\phi_{n}(x_2)\psi_{l+1}(x_3)\Big)^{\intercal},
\end{aligned}
\end{equation}
and $\big\{ \bs \Phi^{x}_{n,l},\, \bs \Phi^{y}_{m,l}, \, \bs \Phi^{z}_{m,n}  \big\}$ stands for the face modal basis functions defined by
\begin{align}
&  \bs \Phi^{x}_{n,l}(\bs x)=\big( 0,\; \; \psi_{n+1}(x_2) \phi_l(x_3),\;\; -\phi_n(x_2) \psi_{l+1}(x_3)  \big)^{\intercal},  \label{eq:hdive1}\\
&  \bs \Phi^{y}_{m,l}(\bs x)=\big(  \psi_{m+1}(x_1) \phi_l(x_3),\;\;0,\;\; -\phi_m(x_1) \psi_{l+1}(x_3)  \big)^{\intercal}, \label{eq:hdive2}\\
&  \bs \Phi^{z}_{m,n}(\bs x)=\big(  \psi_{m+1}(x_1) \phi_n(x_2),\;\; -\phi_m(x_1) \psi_{n+1}(x_2) ,\;\;0 \big)^{\intercal},\label{eq:hdive3}
\end{align}
where $\{\psi_{m+1} \}_{m=1}^{N-1}$ and $\{\phi_{m} \}_{m=1}^{N-1}$ are defined in equation \eqref{eq: psiphi}.
\end{prop}
\begin{proof}
By the identity \eqref{eq:psirelaphi}, one confirms that
\begin{equation}\label{eq:div03did}
\nabla \cdot \bs \chi(\bs x)=0,\quad  \bs \chi=\bs \Phi_{m,n,l}^1, \;\bs \Phi_{m,n,l}^2,\;  \bs \Phi^{x}_{n,l},\;  \bs \Phi^{y}_{m,l},\; \bs \Phi^{z}_{m,n}.
\end{equation}
This ends the proof.
\end{proof}

Correspondingly, the approximation scheme of the three-dimensional curl-curl problem with constant coefficients in equation \eqref{eq: weakform} takes the form: find $\bs u_N \in \mathbb{ X}_{N,3d}^{{\rm div}0}$ such that
\begin{equation}\label{eq: weakform3d}
\mathcal{A}(\bs u_N,\bs v):=(\nabla\times {\bs u}_N,\nabla\times {\bs v})+\kappa ({\bs u}_N,\bs v)=({\bs f},\bs v),\quad \forall \bs v\in \mathbb{ X}_{N,3d}^{{\rm div}0}.
\end{equation}
Let us expand the numerical solution $\bs u_N$ by
\begin{equation*}\label{eq: uexp3d}\begin{split}
\bs u_N(\bs x)=& \sum_{m,n,l=1}^{N-1} \big\{  u_{mnl}^1\bs\Phi_{m,n,l}^1+u_{mnl}^2\bs\Phi_{m,n,l}^2 \big\} \\&+\sum_{n,l=1}^{N-1} u^{x}_{nl} \bs \Phi^{x}_{n,l} + \sum_{m,l=1}^{N-1} u^{y}_{ml} \bs \Phi^{y}_{m,l}  +\sum_{m,n=1}^{N-1} u^{z}_{mn} \bs \Phi^{z}_{m,n} ,
\end{split}\end{equation*}
where $f^i_{\hat m \hat n \hat l}=\big(\bs f,\bs \Phi^i_{\hat m, \hat n, \hat l}\big),\;\; i=1,2$,
and 
\begin{equation*}\label{eq: fexp3d}
\begin{aligned}
&f^{x}_{\hat n \hat l}=\big(\bs f,\bs \Phi^{x}_{\hat n,\hat l}\big),\;\; f^{y}_{\hat m \hat l}=\big(\bs f,\bs \Phi^{y}_{\hat m,\hat l}\big),\;\; f^{z}_{\hat m \hat n}=\big(\bs f,\bs \Phi^{z}_{\hat m,\hat n}\big).
\end{aligned}
\end{equation*}
Remapping the multidimensional variables $\{u^1_{mnl}, u^2_{mnl}, u^x_{nl},u^y_{ml}, u^z_{mn}  \}$ and the source terms $\{f^1_{mnl},$ $f^2_{mnl},$ $f^x_{nl},$ $f^y_{ml}, f^z_{mn}\}$ into a vector form following a natural index ordering as follows:
\begin{equation*}
\begin{aligned}
&\big\{ \big({\rm vec}(u^i)\big)_{p_1}, ({\rm vec}(f^i))_{p_1}  \big\}_{p_1=1,\cdots,(N-1)^3}=\{ u^i_{mnl}, f^i_{mnl} \}_{m,n,l=1}^{N-1}, \;\;i=1,2,\\[5pt]
&\big\{ \big({\rm vec}(u^x)\big)_p, \big({\rm vec}(f^x)\big)_p  \big\}_{p=1,\cdots,(N-1)^2}=\{ u^x_{nl}, f^x_{nl} \}_{n,l=1}^{N-1}, \;\; p=n+(N-1)(l-1),\\[5pt]
&\big\{ \big({\rm vec}(u^y)\big)_p, \big({\rm vec}(f^y)\big)_p  \big\}_{p=1,\cdots,(N-1)^2}=\{ u^y_{ml}, f^y_{ml} \}_{m,l=1}^{N-1}, \;\; p=m+(N-1)(l-1),\\[5pt]
&\big\{ \big({\rm vec}(u^z)\big)_p, \big({\rm vec}(f^z)\big)_p \big\}_{p=1,\cdots,(N-1)^2}=\{ u^z_{mn}, f^z_{mn} \}_{m,n=1}^{N-1}, \;\; p=m+(N-1)(n-1),
\end{aligned}
\end{equation*}
where $p_1=m+(N-1)(n-1)+(N-1)^2(l-1)$ for the first identity in the above.
Let us take $\bs v(\bs x)=\Big\{ \bs\Phi_{\hat m,\hat n,\hat l}^1,\; \bs\Phi_{\hat m,\hat n,\hat l}^2,$ $ \bs \Phi^{x}_{\hat n,\hat l},\; \bs \Phi^{y}_{\hat m,\hat l},\; \bs \Phi^{z}_{\hat m,\hat n}   \Big \}$ in equation \eqref{eq: weakform3d} for $\hat m,\hat n,\hat l=1,\cdots, N-1$, one obtains the following system of linear equations:
\begin{equation}\label{eq:mateq3d}
\begin{aligned}
&A^{11}\, {\rm vec}(u^1)+A^{12}\, {\rm vec}(u^2)={\rm vec}(f^1),\;\;\;\;A^{21} \,{\rm vec}(u^1)+A^{22}\, {\rm vec}(u^2)={\rm vec}(f^2),\\
& A^{xx}\, {\rm vec}(u^x)={\rm vec}(f^x),\;\;\;\;A^{yy}\, {\rm vec}(u^y)={\rm vec}(f^y),\;\;\;\;A^{zz}\,{\rm vec}(u^z)={\rm vec}(f^z),
\end{aligned}
\end{equation}
where
\begin{equation}\label{eq:3dgovmat}
\begin{split}
A^{11}:=&S\otimes M \otimes I+ S\otimes I \otimes M+I \otimes M \otimes S+I \otimes S \otimes M \\
&+2I \otimes I \otimes I+\kappa \big(I \otimes I \otimes M+I \otimes M \otimes I  \big),\\
A^{22}:=&M\otimes S \otimes I+ S\otimes I \otimes M+M \otimes I \otimes S+I \otimes S \otimes M \\
&+2I \otimes I \otimes I+\kappa \big(I \otimes I \otimes M+M \otimes I \otimes I \big),\\
A^{12}=&A^{21}:=S \otimes I \otimes M+I \otimes S \otimes M+I \otimes I \otimes I+\kappa  I \otimes I \otimes M,\\
A^{xx}=&A^{yy}=A^{zz}:=2\Big( M \otimes S+S\otimes M+2I \otimes I +\kappa \big( I \otimes M+M \otimes I \big) \Big).
\end{split}
\end{equation}
%
%
Similar with the two-dimensional case, we construct an auxiliary equation by replacing $S$ by $M^{-1}$ in equation \eqref{eq:3dgovmat} and show that the system of matrix equation  is also fully diagonalizable and the solution can be obtained via analytic expressions. 
\begin{prop}\label{prop: solu3d}
Define an auxiliary system of matrix equations of the original system \eqref{eq:mateq3d}
\begin{equation}\label{eq:mateq3dap}
\begin{aligned}
&\tilde A^{11}\, {\rm vec}(u^1)+\tilde A^{12}\, {\rm vec}(u^2)={\rm vec}(f^1),\;\;\;\tilde A^{21}\, {\rm vec}(u^1)+ \tilde A^{22}\, {\rm vec}(u^2)={\rm vec}(f^2),\\
& \tilde A^{xx}\, {\rm vec}(u^x)={\rm vec}(f^x),\;\;\; \tilde A^{yy}\,{\rm vec}(u^y)={\rm vec}(f^y),\;\;\; \tilde A^{zz}\,{\rm vec}(u^z)={\rm vec}(f^z),
\end{aligned}
\end{equation}
where $\tilde A^{xx}= \tilde A^{yy}= \tilde A^{zz}$ and 
\begin{equation}
\begin{split}
\tilde A^{xx}=&:=2\Big( M \otimes M^{-1}+M^{-1}\otimes M+2I \otimes I +\kappa \big( I \otimes M+M \otimes I \big) \Big),\\
\tilde A^{11}:=&M^{-1}\otimes M \otimes I+ M^{-1}\otimes I \otimes M+I \otimes M \otimes M^{-1}+I \otimes M^{-1} \otimes M \\
&+2I \otimes I \otimes I+\kappa \big(I \otimes I \otimes M+I \otimes M \otimes I  \big),\\
\tilde A^{22}:=&M\otimes M^{-1} \otimes I+ M^{-1}\otimes I \otimes M+M \otimes I \otimes M^{-1}+I \otimes M^{-1} \otimes M \\
&+2I \otimes I \otimes I+\kappa \big(I \otimes I \otimes M+M \otimes I \otimes I \big),\\
\tilde A^{12}=&\tilde A^{21}:=M^{-1} \otimes I \otimes M+I \otimes M^{-1} \otimes M+I \otimes I \otimes I+\kappa I \otimes I \otimes M.
\end{split}
\end{equation}
Here, $M$ is the penta-diagonal matrix defined in equation \eqref{eq: M} and can be diagonalized by \eqref{eq: Mdiag}. 
Let us denote the auxiliary variables $\{ v^1_{ijk}, v^2_{ijk}, v_{jk}^x, v_{ik}^y, v_{ij}^z \}_{i,j,k=1}^{N-1}$ and define for $m,n,l=1,\cdots, N-1$,
\begin{equation}\label{eq: f3d}
\begin{aligned}
& \tilde f^1_{mnl}=\sum_{i,j,k=1}^{N-1} E^{\intercal}_{mi} E^{\intercal}_{nj}E^{\intercal}_{lk} f^1_{ijk}, \quad \tilde f^2_{mnl}=\sum_{i,j,k=1}^{N-1}E^{\intercal}_{mi}E^{\intercal}_{nj}E^{\intercal}_{lk} f^2_{ijk},\\
& \tilde f^x_{nl}=\sum_{i,j=1}^{N-1} E_{ni}^{\intercal} E_{lj}^{\intercal} f^x_{ij},\quad \tilde f^y_{ml}=\sum_{i,j=1}^{N-1} E_{mi}^{\intercal} E_{lj}^{\intercal} f_{ij}^y,\quad \tilde f^z_{mn}=\sum_{i,j=1}^{N-1} E_{mi}^{\intercal} E_{nj}^{\intercal} f^z_{ij}.
\end{aligned}
\end{equation}
The solution $\{u_{mnl}^1,u_{mnl}^2, u_{nl}^x,u_{ml}^y, u_{mn}^z \}$ for $m,n,l=1,\cdots, N-1$  takes the form
\begin{equation}\label{eq: Usolved}
\begin{aligned}
&u^1_{mnl}=\sum_{i,j,k=1}^{N-1} E_{mi}E_{nj}E_{lk} v^1_{ijk},\quad u^2_{mnl}=\sum_{i,j,k=1}^{N-1} E_{mi}E_{nj}E_{lk} v^2_{ijk},\\
&u_{nl}^x=\sum_{j,k=1}^{N-1}E_{nj}E_{lk} v_{jk}^x,\quad u_{ml}^y=\sum_{i,k=1}^{N-1}E_{mi}E_{lk} v_{ik}^y,\quad u_{mn}^z=\sum_{i,j=1}^{N-1}E_{mi}E_{nj} v_{ij}^z,
\end{aligned}
\end{equation}
where
\begin{equation}\label{eq: Vsolved}
\begin{aligned}
 & v_{mnl}^{1}=\dfrac{ \gamma_{mnl}^{22} \tilde f_{mnl}^{1}-\gamma_{mnl}^{12}\tilde f_{mnl}^{2}}{\gamma_{mnl}^{22}\gamma_{mnl}^{11}-\gamma_{mnl}^{12}\gamma_{mnl}^{21}},\qquad    v_{mnl}^{2}=\dfrac{ -\gamma_{mnl}^{21} \tilde f_{mnl}^{1}+\gamma_{mnl}^{11}\tilde f_{mnl}^{2}}{\gamma_{mnl}^{22}\gamma_{mnl}^{11}-\gamma_{mnl}^{12}\gamma_{mnl}^{21}},\\
&     v_{nl}^{x}= \dfrac{\tilde f^{x}_{nl}}{ d_n/d_l+d_l/d_n+2+\kappa (d_n+d_l)},\\
&   v_{ml}^{y}= \dfrac{\tilde f^{y}_{ml}}{d_m/d_l+d_l/d_m+2+\kappa (d_m+d_l) },\\
&    v_{mn}^{z}= \dfrac{\tilde f^{z}_{mn}}{d_m/d_n+d_n/d_m+2+\kappa (d_m+d_n)  },
\end{aligned}
\end{equation}
with 
\begin{equation}\label{eq: cmnl}
\begin{aligned}
& \gamma_{mnl}^{11}= d_n/ d_l+d_m/ d_l +d_n/ d_m+ d_m/d_n +2 +\kappa (d_m+d_n),\\
& \gamma_{mnl}^{12}=d_m /d_l+d_m/ d_n +1 + \kappa d_m ,\;\;\; \gamma_{mnl}^{21}=d_m/ d_l+d_m/ d_n +1 + \kappa d_m,\\
& \gamma_{mnl}^{22}= d_l/d_n +d_m/ d_l + d_l/d_m +d_m/ d_n  +2+\kappa ( d_m +d_l) .
\end{aligned}
\end{equation}
\end{prop}
\begin{proof}
The  proof is similar with that of Proposition \ref{prop: aux2d}. For the sake of conciseness, we skip the detail. 
\end{proof}
It can be observed that the face modes $\big \{ {\rm vec}(u^x),{\rm vec}(u^y), {\rm vec}(u^z)  \big \}$ are fully decoupled with the interior modes $\big\{ {\rm vec}(u^1),{\rm vec}(u^2) \big \}$ and the governing equations for the face modes are exactly the same with the two-dimensional case in equation \eqref{eq: meq2d}. Thus, we could focus on the solution of the interior modes.  With a slight abuse of notation, we denote $A$ and $\tilde A$, respectively, by
\begin{equation}\label{eq: AtildeA}
A=\begin{bmatrix}
A^{11} & A^{12}  \\
A^{21} & A^{22} 
\end{bmatrix},\quad
\tilde A=\begin{bmatrix}
 \tilde A^{11} &  \tilde A^{12}  \\
\tilde A^{21} & \tilde A^{22}  
\end{bmatrix}.
\end{equation}

Once again, we investigate theoretically the spectrum of $\tilde A^{-1}A$ and the following result holds.

\begin{thm}\label{thm:3dinvarspace}
The dimension of the invariant subspace of the $2(N-1)^3$-by-$2(N-1)^3$ matrix $\tilde A^{-1} A$  with respect to the eigenvalue $1$ is $2(N-3)^3$.
\end{thm}
\begin{proof}
The proof is similar with Theorem \ref{thm:2dinvarspace}. It suffices to find $2(N-3)^2$ linearly independent solutions $\{{\rm vec}(u_1),{\rm vec}(u_2)\}$ such that 
\begin{equation}\label{eq:proof3d1}
\begin{aligned}
&A^{11}\, {\rm vec}(u^1)+A^{12}\, {\rm vec}(u^2)=\tilde A^{11}\, {\rm vec}(u^1)+\tilde A^{12}\, {\rm vec}(u^2),\\
&A^{21} \,{\rm vec}(u^1)+A^{22}\, {\rm vec}(u^2)=\tilde A^{21} \,{\rm vec}(u^1)+\tilde A^{22}\, {\rm vec}(u^2).
\end{aligned}
\end{equation}
Let us introduce the auxiliary variables $\{ \bar v^1_{mnl}, \bar v^2_{mnl}  \}$ such that 
\begin{equation*}\label{eq:proof3d2}
u^1_{mnl}=\sum_{i,j,k=1}^{N-1} M_{mi}M_{nj}M_{lk} \bar v^1_{ijk},\quad u^2_{mnl}=\sum_{i,j,k=1}^{N-1} M_{mi}M_{nj}M_{lk} \bar v^2_{ijk}.
\end{equation*}
Inserting the above relation into equation \eqref{eq:proof3d1}, using the fact that the elements in the first $(N-3)$ columns of $SM-I$ and the first $(N-3)$ rows of $MS-I$ are zeros and choosing $\{ \bar v^1_{mnl}, \bar v^2_{mnl} \}$ such that
\begin{equation*}\label{eq:proof3d3}
 \bar v^i_{mnl}=
 \begin{cases}
x^i_{mnl}, \;\; & 1\leq m,n,l\leq N-3,\\
0, &{\rm otherwise},
 \end{cases},\quad i=1,2. 
\end{equation*}
where $\{ x^i_{mnl} \}^{i=1,2}_{1\leq m,n,l\leq N-3}$ are $2(N-3)^3$ free variables. 
One can readily show that $\{ u^1_{mnl}, u^2_{mnl} \}$ satisfies the linear system \eqref{eq:proof3d1}, which ends the proof. 

\end{proof}

Thus, one can expect that the auxiliary problem in Proposition \ref{prop: solu3d}, when combined with a preconditioned Krylov subspace iterative method, gives rise to an efficient solution algorithm for the three dimensional curl-curl problem with constant or variable coefficients. The solution procedure resembles Algorithm \ref{ag:ag1}.  For the sake of conciseness, we skip the details of this algorithm.  

\begin{rem}{\em 
The overall computational cost for the three-dimensional problem is dominated by the tensor-matrix multiplication, which is only a small multiple of $N^4$ operations, as opposed to $O(N^6)$ operations for assembled global system.}
\end{rem}

%

\section{Representative numerical examples}
In this section, we provide ample numerical examples to verify the accuracy and efficiency of our proposed fast divergence-free spectral algorithm. We conduct convergence tests employing both two- and three-dimensional contrived solutions, with both constant and variable coefficients.  We present the convergence rates in both $L^2$- and $H({\rm curl})$-norms. Additionally, we also subject the method to challenging numerical tests on indefinite systems with highly oscillatory solutions induced by inhomogeneous point sources.  In what follows, we adopt a stopping threshold $\varepsilon=10^{-12}$ for the proposed iterative method, without stated otherwise. 

\begin{example}\label{ex: ex1} {\bf (Convergence test for curl-curl problem with constant coefficients in 2D).} 
We first consider \eqref{eq:system} with the following manufactured solution
\begin{equation}\label{examp1}
\bs u(x,y) = \nabla \times \Big(\big(1-x^2 \big)^3 \big(1-y^2 \big)^3\big(1+\sin(\pi x)\cos(\pi y)\big)\exp\big(\sin(\pi x)\cos(\pi y) \big)\Big).
\end{equation}
Accordingly, the source term $\bs f(x,y)$ in equation \eqref{eq:system} is chosen such that the analytic expression \eqref{examp1} satisfies equation \eqref{eq:system}.
\end{example}
 \begin{figure}[htbp]
\begin{center}
  \subfigure[ Errors vs $N$]{ \includegraphics[scale=.42]{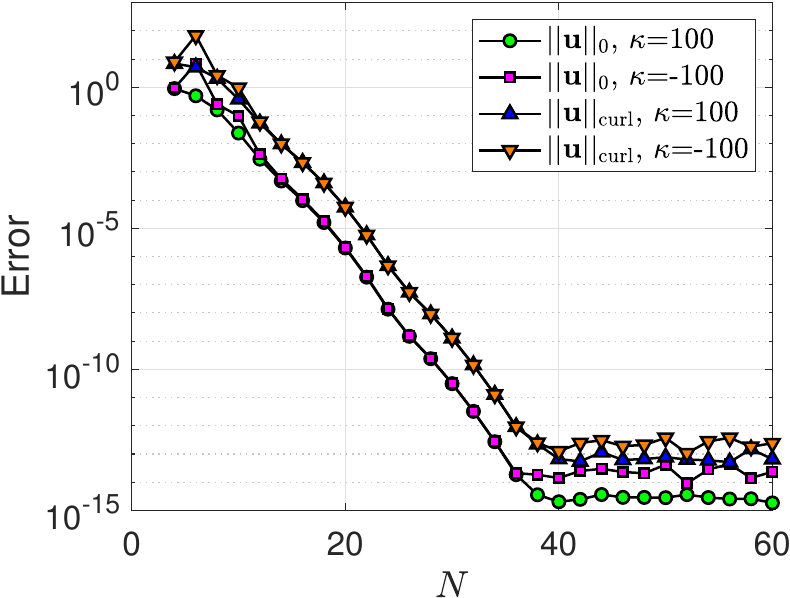}}\qquad 
  \subfigure[Iteration vs no iteration ]{ \includegraphics[scale=.42]{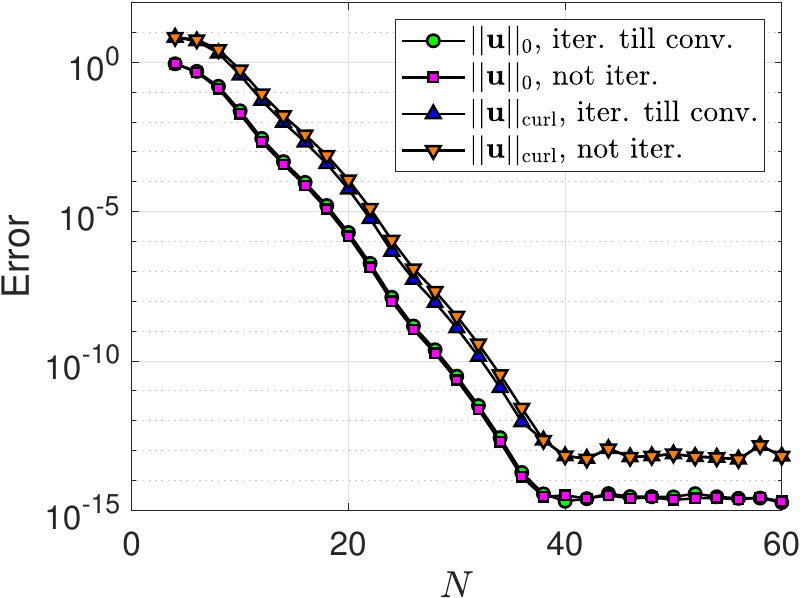}}\quad    
  \vspace{-6pt} \caption{\small Example \ref{ex: ex1}: convergence tests of the curl-curl problem with constant coefficients in 2D. (a) $L^2$- and $H({\rm curl})$-errors of $\bs u$ versus polynomial order $N$ for $\kappa=\pm 100$; (b)  compare the errors between the solution obtained from the proposed method and those obtained directly from the auxiliary problem in Proposition \ref{prop: aux2d}. } 
   \label{fig_err_N2d}
\end{center}
\end{figure}

In Fig.\,\ref{fig_err_N2d} (a), we plot the discrete $L^2$- and $H({\rm curl})$-errors on the semi-log scale against various $N$ using the proposed fast divergence-free spectral algorithm. We observe the expected exponential convergence of the proposed method for both positive and negative parameters $\kappa$ as $N$ increases.   Fig.\,\ref{fig_err_N2d} (b) compares the numerical errors of the current method with those obtained from  directly solving the auxiliary problem stated in Proposition \ref{prop: aux2d}, This comparison indicates that the auxiliary problem provides numerical solutions with almost the same accuracy as the iterative algorithm for rather smooth solutions. Thus, the auxiliary problem can serve as a good approximation for the original problem, which is in agreement with previous discussions in Section \ref{sec2.3}.

\begin{table}[h!tbp]
\caption{\small Example \ref{ex: ex1}: iteration numbers and numerical errors with various $\kappa$.}
\vspace{-6pt}\begin{tabular}{cccccccccccc}
\hline & \multicolumn{2}{c}{$\kappa=1$ } & & \multicolumn{2}{c}{ $ \kappa=100$ }  & & \multicolumn{2}{c}{$\kappa=-1$ }& & \multicolumn{2}{c}{$\kappa=-100$ } \\
\cline { 2 - 3 } \cline { 5 - 6 }\cline { 8 - 9 } \cline { 11 - 12 }$N$&  iter & Error & & iter &  Error & & iter & Error & & iter & Error \\
\hline 
   12 & 13 & 5.27e-2 & & 12 & 5.31e-2 & & 13 & 5.27e-2 & & 16 & 5.91e-2 \\
  20 & 9 & 5.61e-5 & & 9 & 5.62e-5 & & 9 & 5.61e-5 & & 10 & 5.63e-5 \\
  28 & 4 & 9.10e-9 & & 4 & 9.10e-9 & & 4 & 9.10e-9 & & 4 & 9.11e-9 \\
  36 & 1 & 9.21e-13 & & 1 & 9.21e-13 & & 1 & 9.20e-13 & & 0 & 9.26e-13 \\
  40 & 0 & 6.72e-14 & & 0 & 6.62e-14 & & 0 & 6.57e-14 & & 0 & 1.24e-13 \\
\hline
\end{tabular}\label{tableexp1}
\end{table}

Next, we provide numerical evidence for one prominent feature of the proposed method, namely, the iteration count will decrease as the polynomial order increases. We tabulate in Table \ref{tableexp1} the $H({\rm curl};\Omega)$-errors and the corresponding iteration numbers obtained by Algorithm \ref{ag:ag1} with various $\kappa$ and $N$. Notably, as $N$ increases,  the iteration count decreases, and remarkably, it drops to $0$ when $N=40.$ We then set $\kappa=1$ and further increase the polynomial order, comparing the iteration numbers and CPU time of the proposed algorithm with those obtained without any preconditioner or with Jacobi preconditioner, see Table \ref{table:3}.  It can be observed that the proposed method significantly outperforms the other two methods in terms of both iteration numbers and CPU time. Unlike the algorithms without a preconditioner or with a Jacobi preconditioner, whose iteration numbers increase with the polynomial order, the iteration numbers of the proposed method remain at 0. This occurs because the proportion of linearly independent eigenvectors of $\tilde A^{-1} A$ with respect to 0 is $(N-3)^2/(N-1)^2$, which approaches $100\%$ as $N$ gets larger.

\begin{table}[h!tbp]
\caption{\small Example \ref{ex: ex1}: iteration numbers and CPU time for large $N$.}
\setlength{\tabcolsep}{1.3mm}{
\begin{tabular}{ccccccccc}
\hline & \multicolumn{2}{c}{No preconditioner } & & \multicolumn{2}{c}{  Jacobi preconditioner }  & & \multicolumn{2}{c}{Current method } \\
\cline { 2 - 3 } \cline { 5 - 6 }\cline { 8 - 9 }$N$ &  iter & CPU time (s) & & iter &  CPU time (s) & & iter &  CPU time (s) \\
\hline 50 & 237 & 1.34 & & 158 & 0.94& & 0 & 1.82e-3 \\
  100 & 523 & 30.62 & & 248 & 7.42 & & 0 & 6.76e-3 \\
  500 & Not conv. & N.A. & & 622 & 1040.21 & & 0 & 0.29\\
  1000 & Not conv. & N.A. & & Not conv. & N.A. & & 0 & 1.58\\
\hline
\end{tabular}}\label{table:3}
\end{table}

\begin{example}\label{ex: ex3} {\bf (Curl-curl problem with point source in 2D).} 
We consider equation \eqref{eq:system} with Gaussian point source
\begin{equation*}\label{examp2}
\bs f=\nabla \times \Big( {\rm exp}\Big({-\frac{(x+0.5)^2+y^2}{\sigma^2}}\Big) + {\rm exp}\Big({-\frac{(x-0.5)^2+y^2}{\sigma^2}} \Big) \Big),
\end{equation*}
where $\sigma=0.01.$  
\end{example}

    \begin{figure}[htbp]
\begin{center}
  \subfigure[$\kappa=-100$ ]{ \includegraphics[scale=.22]{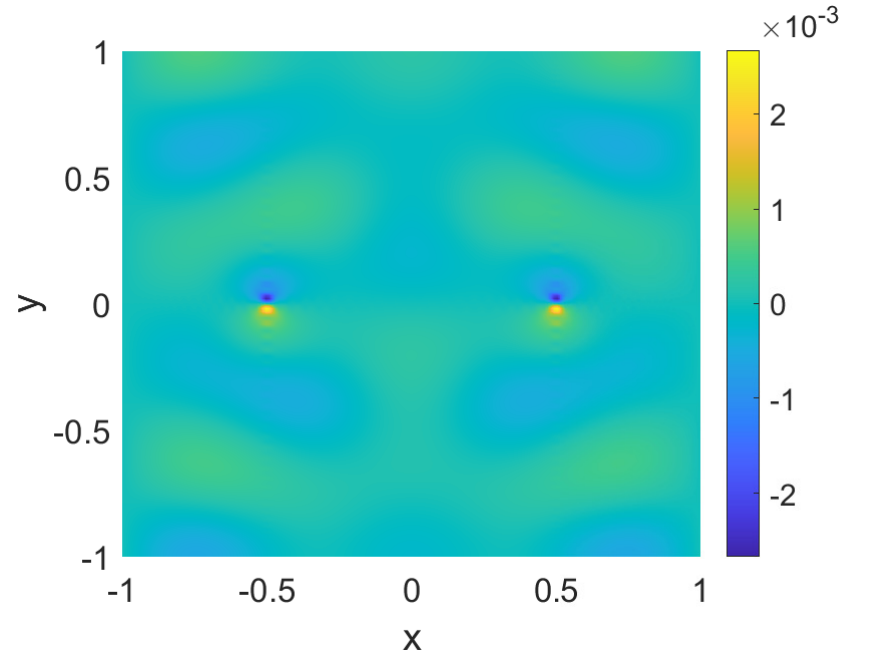}}\hspace{-9pt}
  \subfigure[$\kappa=-400$ ]{ \includegraphics[scale=.22]{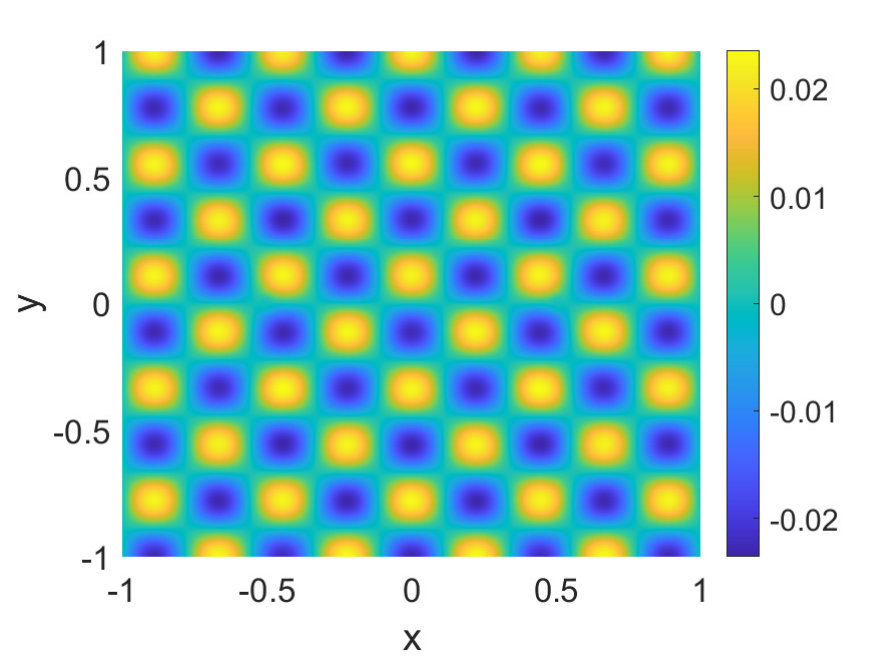}}\hspace{-9pt}
   \subfigure[$\kappa=-2500$ ]{ \includegraphics[scale=.22]{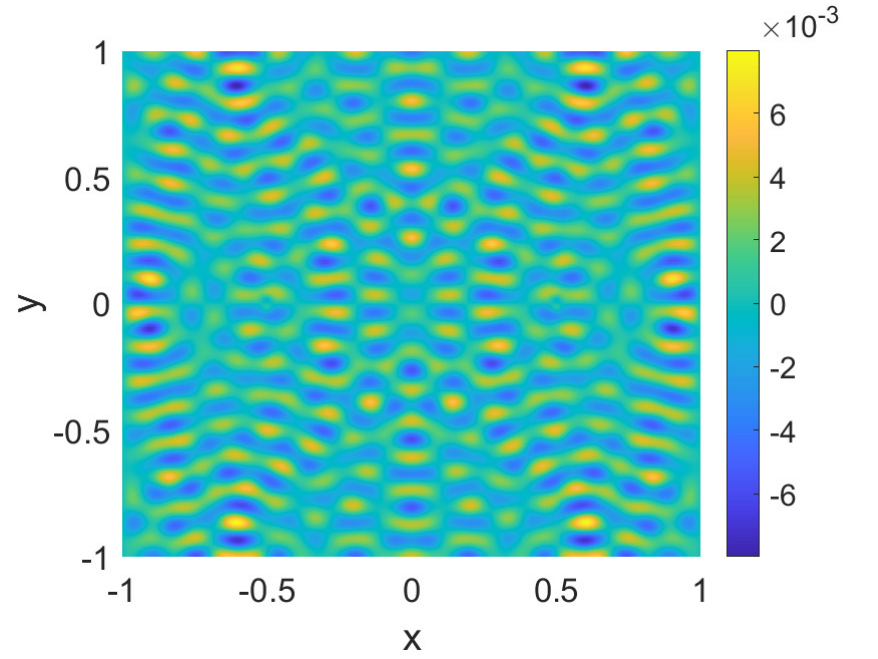}}\hspace{-9pt}
  \subfigure[$\kappa=-10000$ ]{ \includegraphics[scale=.22]{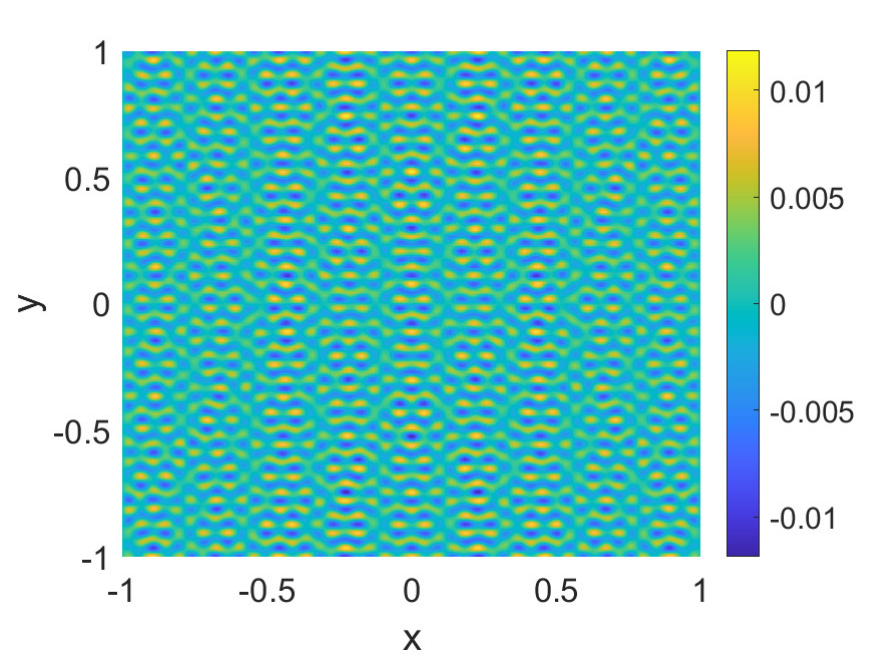}}
 \caption{\small  Example \ref{ex: ex3}: The profiles of the $x$-component of the numerical solution versus different $\kappa$. These results are obtained with a fixed polynomial order $N=128.$ } 
   \label{figsource1}
\end{center}
\end{figure}

In this case, the analytic solution is unavailable and numerical solution is obtained via the proposed fast divergence-free spectral algorithm with polynomial order $N=128$. We challenge our algorithm via highly indefinite system with highly oscillatory solutions by choosing the parameter $\kappa$ as large negative values $-100,-400,-2500,$ $-10000$, see Figure \ref{figsource1}. It can be seen that the proposed method provides good approximations, even for solutions with highly oscillation. We also tabulate in Table \ref{iterationnum2} the number of iterations of the proposed method as a function of the polynomial order, where we take $\kappa=-10000.$ Astonishingly, It can be observed that the number of iterations remains small and decreases from $35$ to $15$ as the polynomial order increases from $120$ to $520$, demonstrating the superiority of the proposed method for curl-curl problem with highly oscillatory solutions.  

\begin{table}[h!tbp]
\begin{center}
\caption{{\small Example \ref{ex: ex3}: iteration numbers for various $N$ with $\kappa=-10000$.}}\small
\centering{
\setlength{\tabcolsep}{2.8mm}{
\vspace{-6pt}\begin{tabular}{c c c c c c c c c c c c}\hline
$N$&120&160&200&240&280&320&360&400&440&480&520 \\
\hline \hline
iter    &35&32&35&34&37&34&32&26&22&18&15\\
   \hline
\end{tabular}}}\label{iterationnum2}
\end{center}
\end{table}

\begin{example}\label{ex: ex4}{\bf (Time-dependent Maxwell's equations with variable coefficients in 2D).} Consider the following time-dependent problem with variable coefficients in $\Lambda^2$ for $t\in(0,T]:$
\begin{equation} \label{eq:}
\begin{cases}
\dfrac{\partial \bs B}{\partial t}+\nabla \times \bs E=0 ,\\[5pt]
\dfrac{\partial}{\partial t}(\epsilon_0 \epsilon_r \bs E) - \frac{1}{\mu_0}\nabla \times \bs B=- \bs J ,
\end{cases}
\end{equation}
where $\bs E$ is the electric field and $\bs B$ is the magnetic field taking the form
\begin{equation*}
\bs E=(0,0, E_3(x,y,t))^{\intercal},\quad \bs B=(B_1(x,y,t),B_2(x,y,t),0)^{\intercal}.
\end{equation*}
Here, $\epsilon_0$ and $\mu_0$ are the permittivity and permeability in vaccum, respectively, and $\epsilon_r$ is the relative permittivity. The electric current density $\bs J$ is chosen as a point source given by
\begin{equation*}\label{examp4}
\bs J=(0,0,J_3(x,y))^{\intercal},\quad J_3(x,y) = \exp\Big(-\frac{(x+0.8)^2+y^2}{\sigma^2}\Big).
\end{equation*}
Here, we choose  $\epsilon_0 = 1,$  $\mu_0 = 1$, $\sigma = 0.04,$ and
\begin{equation*}\label{examptime2}
\epsilon_r(x,y) = 1.5+0.5\tanh\Big(  \frac{0.16+0.07\sin\big(6(\theta+\pi/4) \big)-x^2-y^2}{\lambda}      \Big),
\end{equation*}
where $\theta(x,y)$ is the angle of point $\bs x$ and $\lambda$ is set to be $0.05$. The boundary condition is prescribed as
\begin{equation*}\label{eq:Maxwellbd}
\bs n \cdot \bs B=0, \quad \bs n \times (\nabla \times \bs B)=0\;\; {\rm on}\; \partial \Lambda^2.
\end{equation*}
\end{example}

Let us partition the time interval $(0,T]$ using uniform time step size $\tau$ by
\begin{equation*}\label{eq: tpartition}
t_n=n\tau,\;\; n=0,1,\cdots, N_t, \quad \tau=T/N_t.\
\end{equation*}
and denote $\chi^n$ the numerical approximation of an arbitrary varaibe $\chi $ at time step $n$, corresponding to the time $t_n.$ In the temporal direction, we adopt the Crank-Nicolson (CN) temporal discretization
\begin{equation} \label{eq: CN2d}
\begin{cases}
\dfrac{ \bs B^{n+1}-\bs B^{n}}{\tau}+\nabla \times (\dfrac{\bs E^{n+1}+\bs E^n}{2} )=0 ,\\[5pt]
\epsilon_0 \epsilon_r \dfrac{\bs E^{n+1}-\bs E^n}{\tau} - \dfrac{1}{\mu_0}\nabla \times (\dfrac{\bs B^{n+1}+\bs B^n}{2})=- \bs J,
\end{cases}
\end{equation}
which can be reformulated into a curl-curl problem of $\bs B^{n+1}$. In the spatial direction, we use the proposed fast divergence-free spectral algorithm to discretize $\bs B^{n+1}$. 

We depict the temporal sequence of snapshots (see Fig.\,\ref{figs: Circle} (b)-(d)) of $B_2$ at time $t=0.31,0.71,0.81$. It can be seen that the proposed method can accurately simulate the the interaction of the electromagnetic waves with the heterogeneous medium  (see Fig.\,\ref{figs: Circle} (a)) and there is no numerical oscillations occur for long time simulations. Moreover, the average number of iterations is only around 36, indicating that the auxiliary problem serves as an effective preconditioner for Maxwell's equations with variable coefficients. 

\begin{figure}[htbp]
\begin{center}
    \subfigure[Field phase]{ \includegraphics[scale=.16]{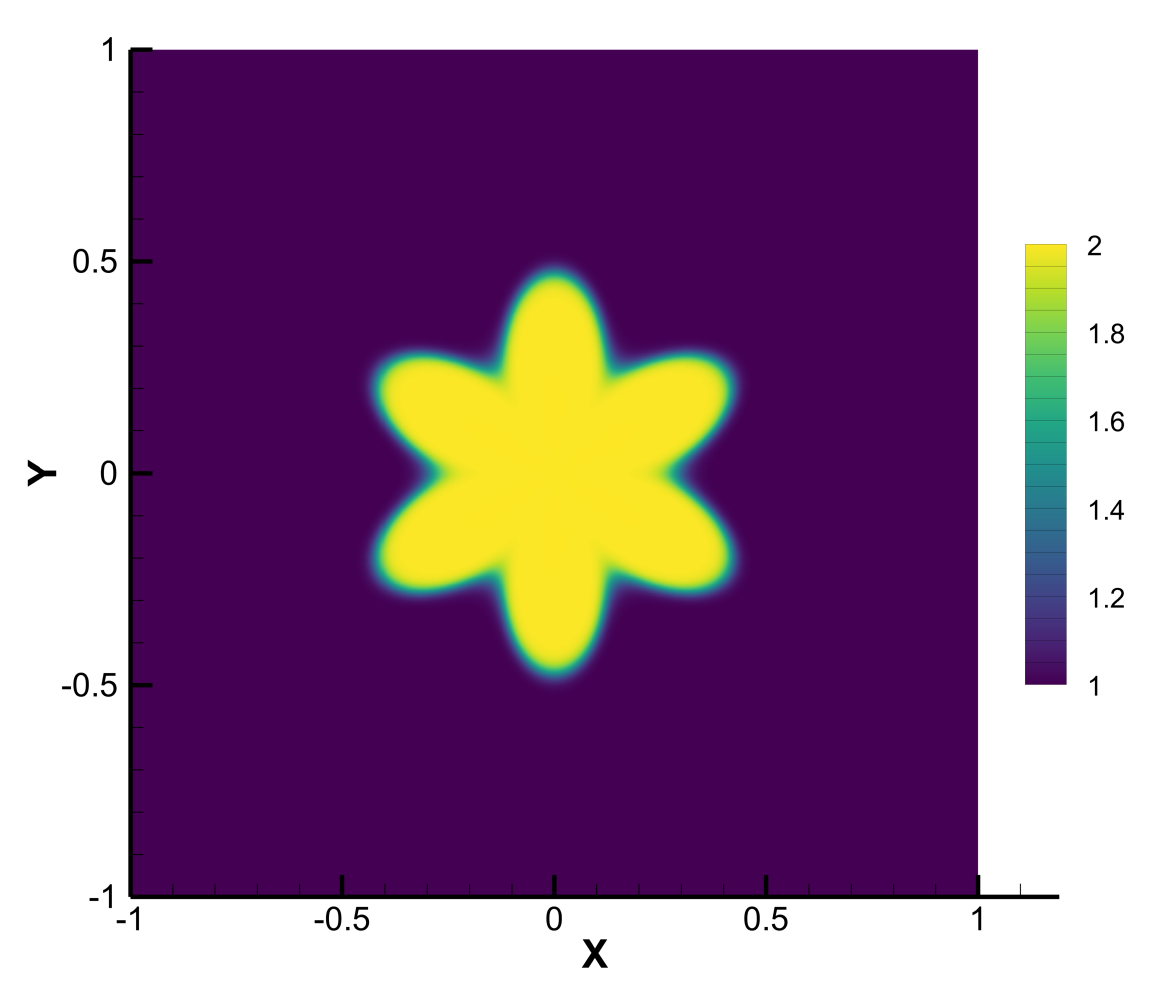}}
  \subfigure[t=0.31]{ \includegraphics[scale=.16]{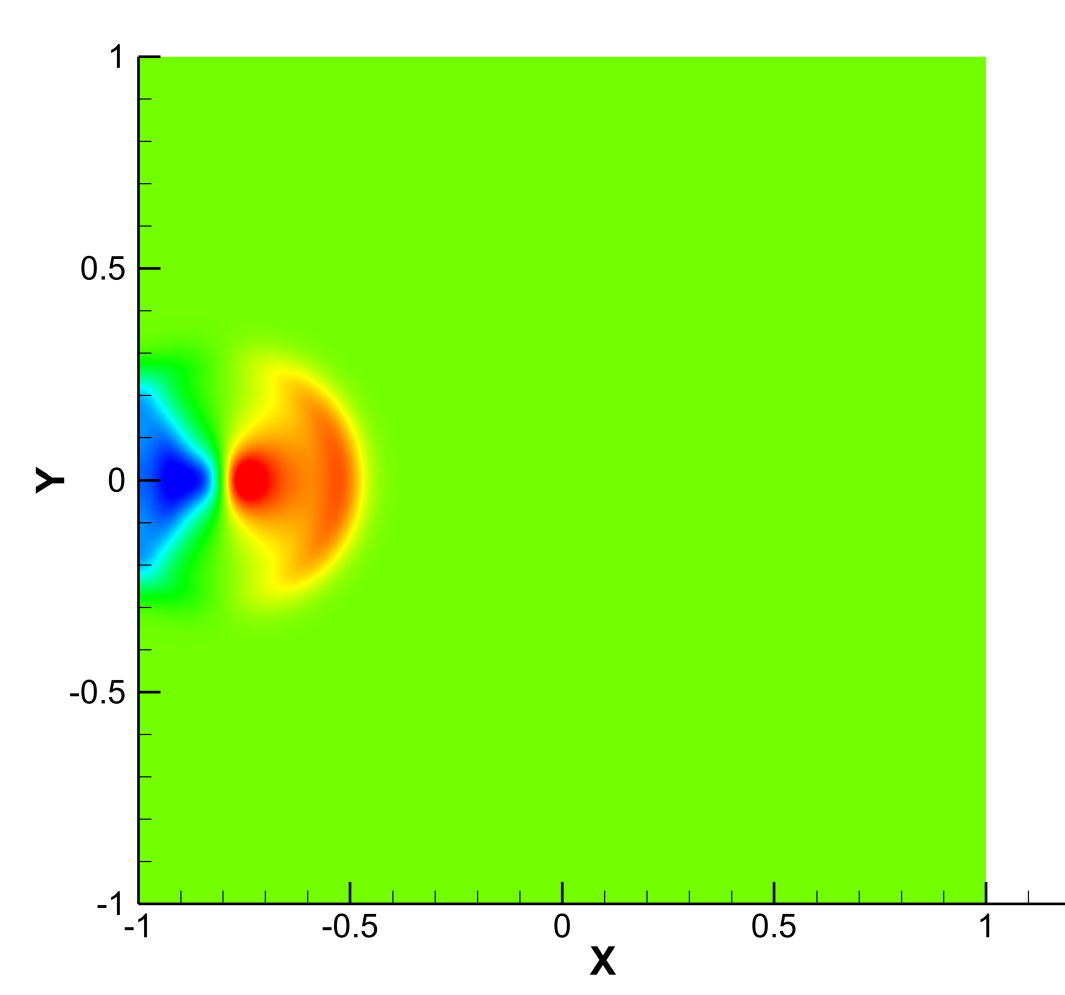}}
 \subfigure[t=0.71]{ \includegraphics[scale=.16]{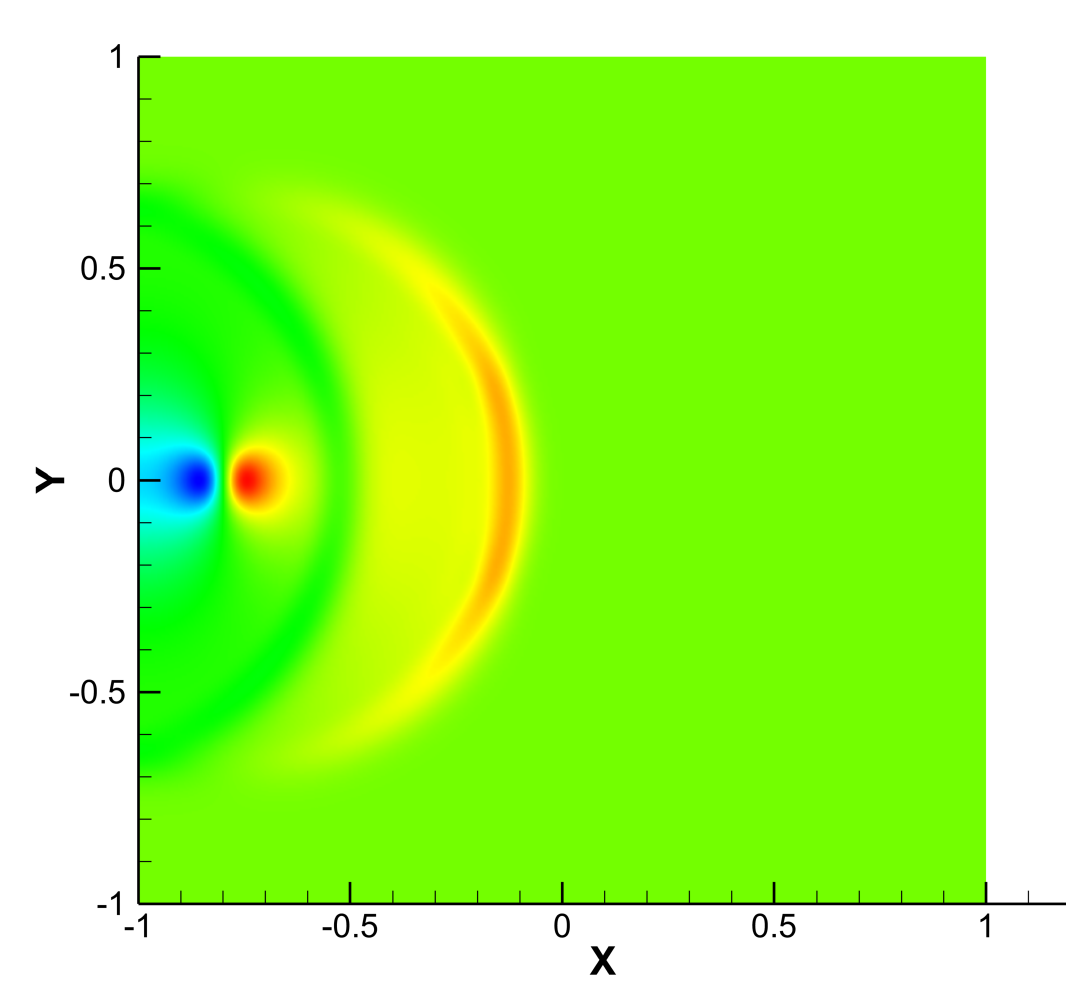}} 
  \subfigure[t=0.81]{ \includegraphics[scale=.16]{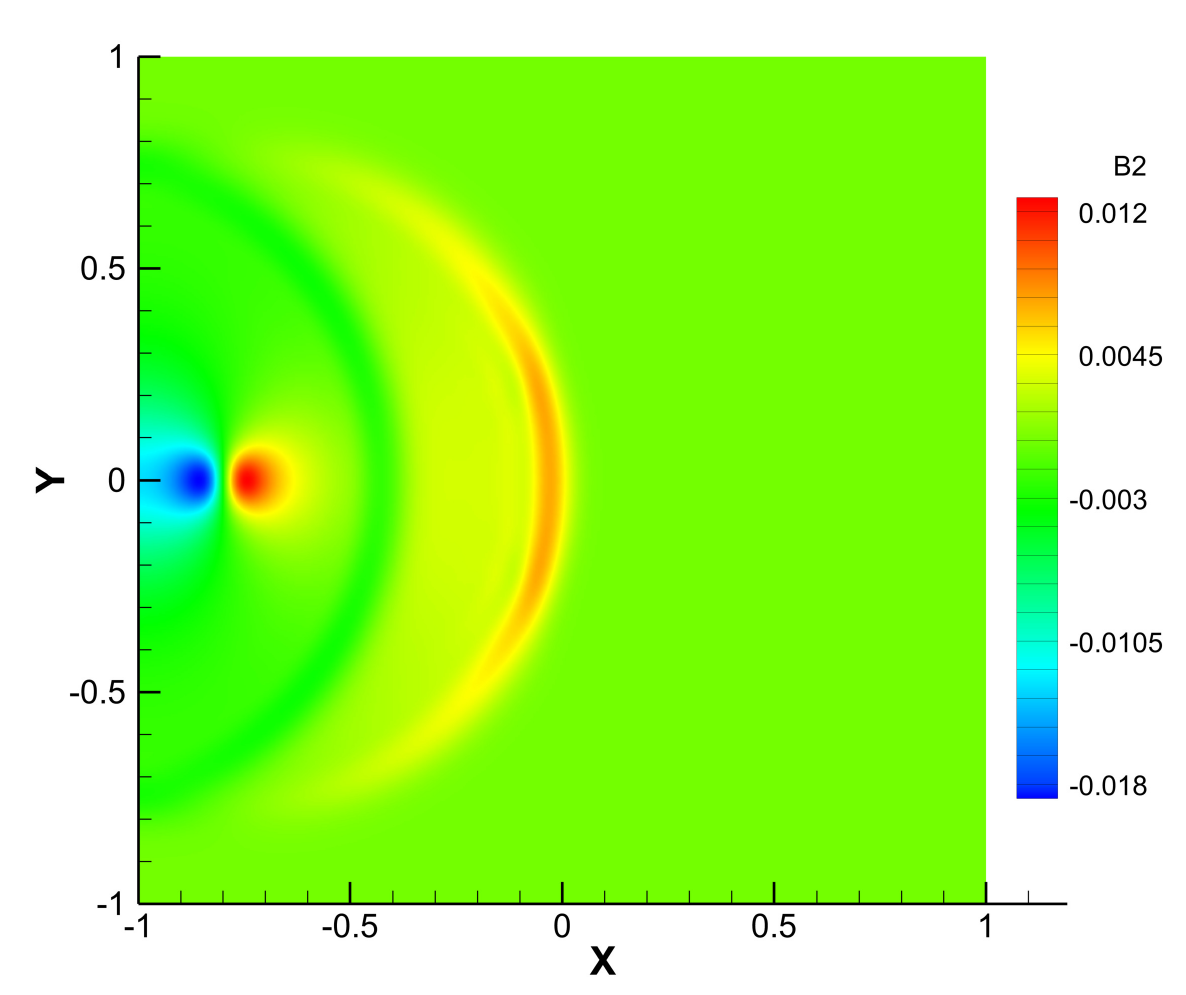}} 
    \caption{Example \ref{ex: ex4}: electromagnetic wave propagation in inhomogeneous medium. (a) Profile of the relative permittivity $\epsilon_r$; (b)-(d) snapshots of $ B_2$ at $t=0.31,0.71,0.81$. The simulations are obtained with $N=250$ and $\tau=0.01.$} 
   \label{figs: Circle}
\end{center}
\end{figure}


\begin{example}\label{ex: ex5} {\bf (Convergence test for curl-curl problem with constant coefficients in 3D).} 
Next, we consider the curl-curl problem \eqref{eq:system} in three dimensions with the manufactured solution
\begin{equation*}\label{exactsolu3d}
\bs u(x,y,z) = \nabla \times \Big(\big(1-x^2\big)^3 \big(1-y^2 \big)^3 \big(1-z^2 \big)^3\big(1+w(x,y,z) \big)\exp(w(x,y,z))\Big)(1,1,1)^{\intercal},
\end{equation*}
where $w(x,y,z)=\sin(\pi x)\sin(\pi y)\sin(\pi z)$, and the source term $\bs f(x,y,z)$ in equation \eqref{eq:system} is chosen such that the analytic expression \eqref{exactsolu3d} satisfies equation \eqref{eq:system}.
\end{example}

  \begin{figure}[htbp]
\begin{center}
  \subfigure[Errors vs $N$]{ \includegraphics[scale=.4]{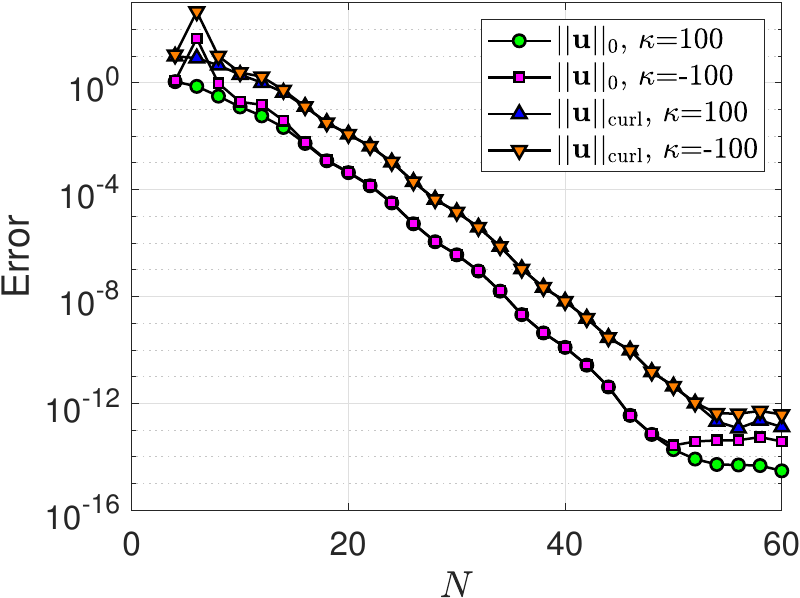}}\qquad 
  \subfigure[Iteration numbers vs $N$]{ \includegraphics[scale=.4]{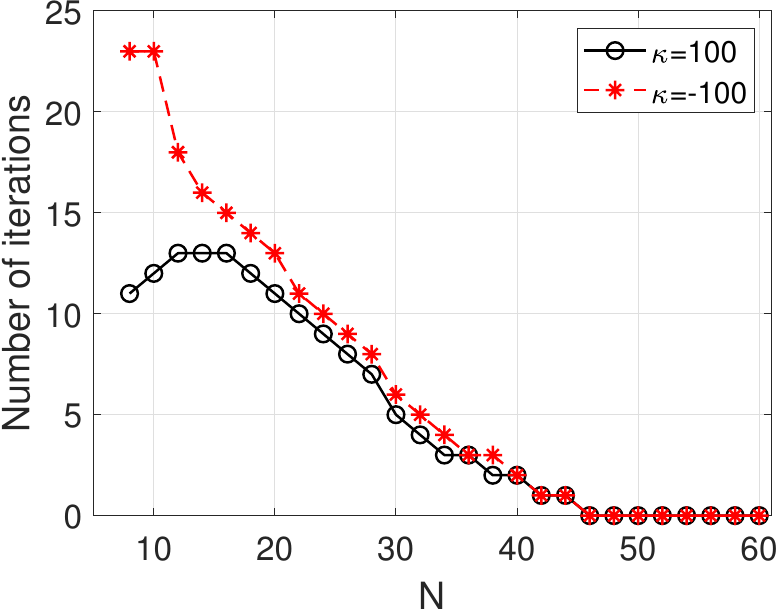}}\quad     
  \caption{\small  Example \ref{ex: ex5}: convergence tests of the curl-curl problem with constant coefficients in 3D. (a) $L^2$- and $H({\rm curl})$-errors of $\bs u$; (b) the number of iterations of the proposed method. The simulations are obtained with $\kappa=\pm 100$.} 
   \label{exam5err}
\end{center}
\end{figure}

In Fig.\,\ref{exam5err}\,(a), we plot the discrete $L^2$- and $H({\rm curl})$-errors on the semi-log scale of the proposed fast divergence-free spectral algorithm for \eqref{eq:system} with \eqref{exactsolu3d}  versus $N$. Observe that the errors decay exponentially in terms of $N$ increases for both positive and negative parameters $\kappa$. We plot the iteration numbers of the proposed method versus $N$, see Fig.\,\ref{exam5err}\,(b). Again, for a given smooth function, we can see that the iteration number reduces to $0$ as $N$ increases, which is in agreement with the analysis in Theorem \ref{thm:3dinvarspace}.

\begin{example}\label{ex: ex6}{\bf (Time-dependent Maxwell's equations with point source in 3D).} Consider the time-dependent Maxwell's equations as in Example \ref{ex: ex4}, 
where
\begin{equation*}
\bs E=(E_1(\bs x,t),E_2(\bs x,t), E_3(\bs x,t))^{\intercal},\quad \bs B=(B_1(\bs x,t),B_2(\bs x,t),B_3(\bs x,t))^{\intercal},
\end{equation*}
and the electric current density $\bs J=(J_1,0,0)^{\intercal}$ is given by
\begin{equation}\label{examp6}
 J_1(x,y,z) = \exp\Big(-\frac{x^2+y^2+(z+0.5)^2}{\sigma^2}\Big)+\exp\Big(-\frac{x^2+y^2+(z-0.5)^2}{\sigma^2}\Big).
\end{equation}
Here, we choose $\sigma = 0.05,$  $\epsilon_0 =\mu_0= 1,$ and  $\epsilon_r = 1.$
\end{example}
Once again, we resort to the Crank-Nicolson scheme for temporal discretization and the proposed fast divergence-free spectral algorithm for spatial discretization of the resultant curl-curl problem of $\bs B$ in three dimensions. In these simulations, we prescribe the polynomial order $N=100$ and the time step size $\tau= 0.02$. It can be observed that the proposed method can accurately simulate the the interaction of the electromagnetic waves induced by inhomogeneous point sources, see Fig.\,\ref{figsource3d}. Moreover, the average number of iterations is merely $8.21$, demonstrating that the proposed method is highly efficient for long time 3D simulations. 

  \begin{figure}[htbp]
\begin{center}
  \subfigure[$t=0.1$ ]{ \includegraphics[scale=.16]{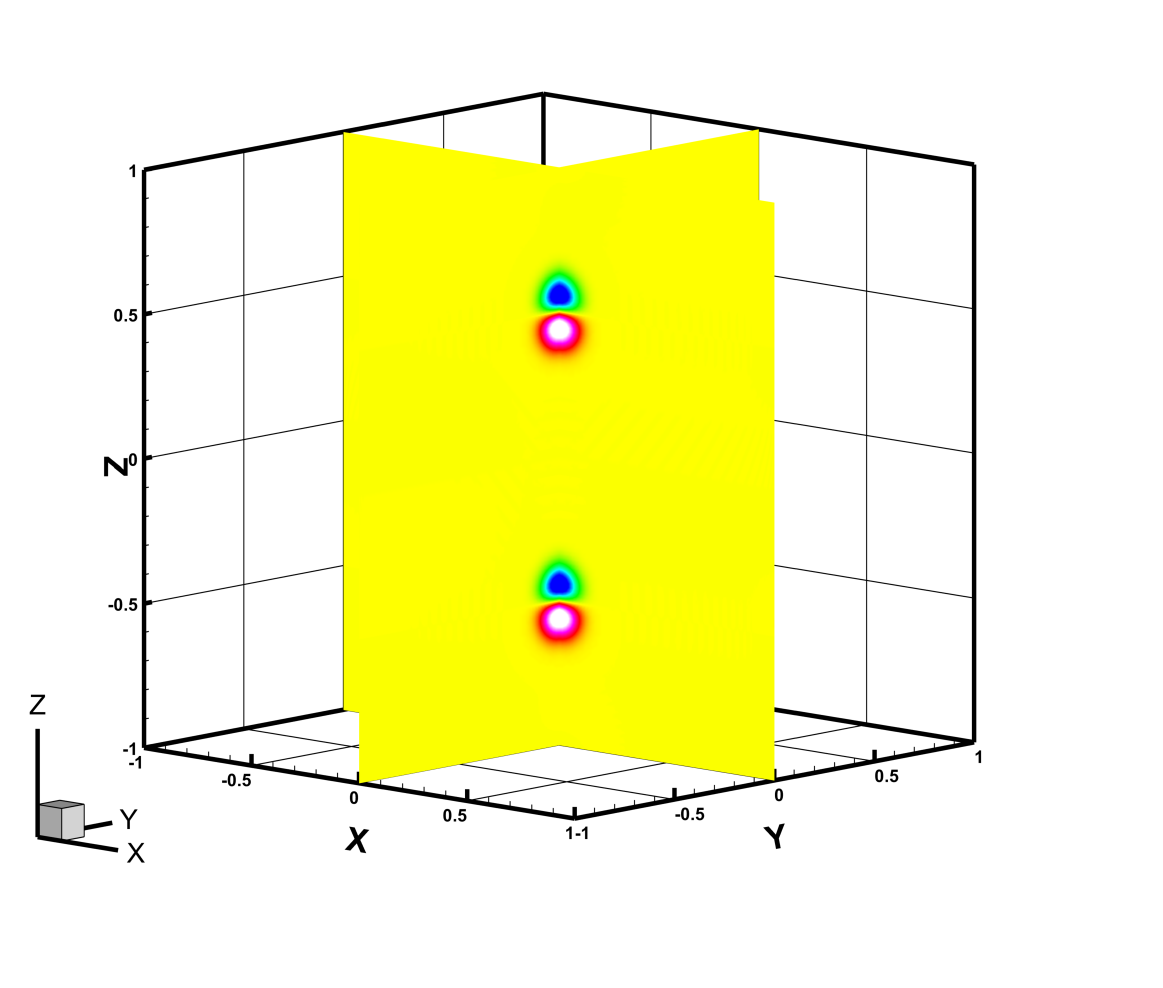}}\hspace{-8pt}
  \subfigure[$t=0.3$ ]{ \includegraphics[scale=.16]{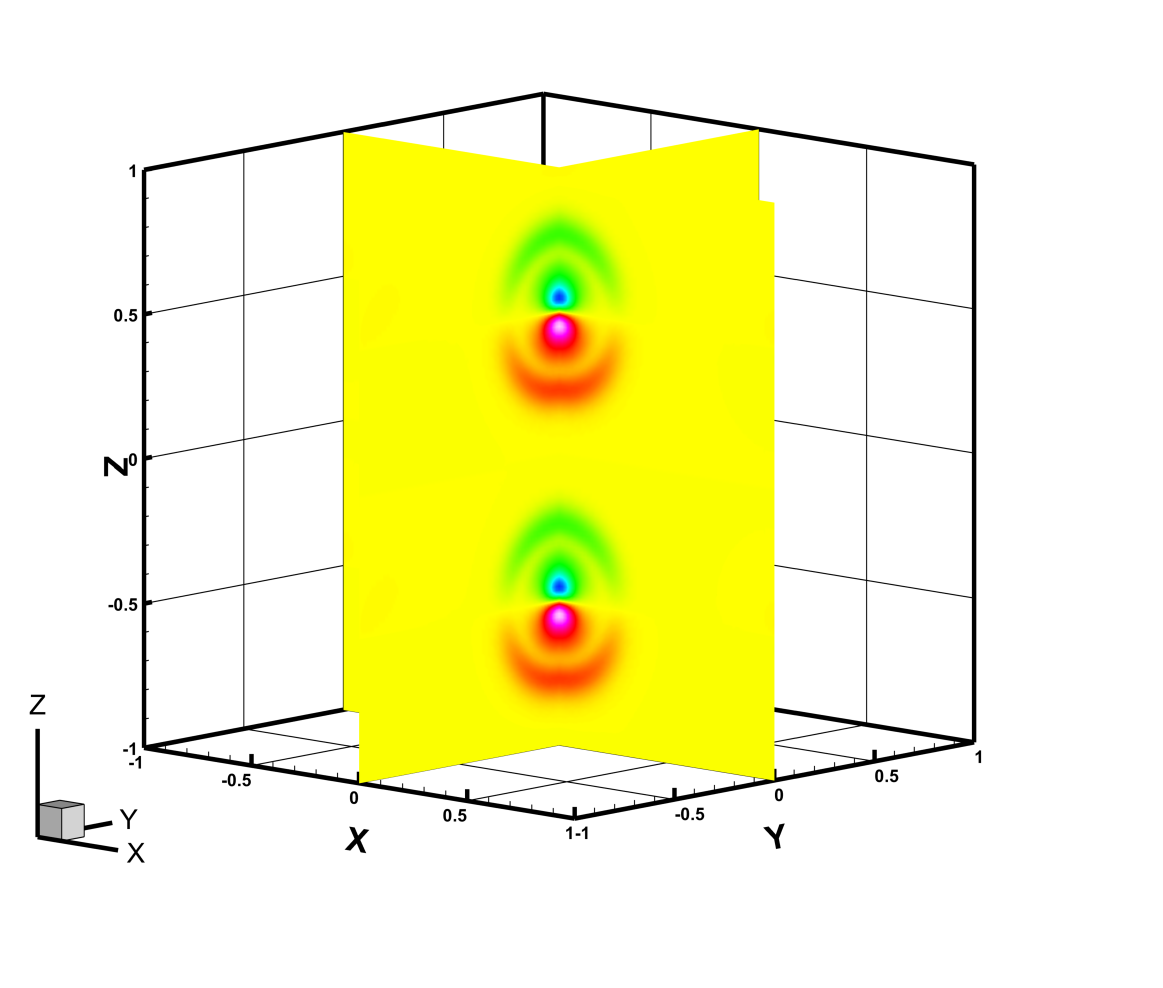}}\hspace{-8pt}
   \subfigure[$t=0.46$ ]{ \includegraphics[scale=.16]{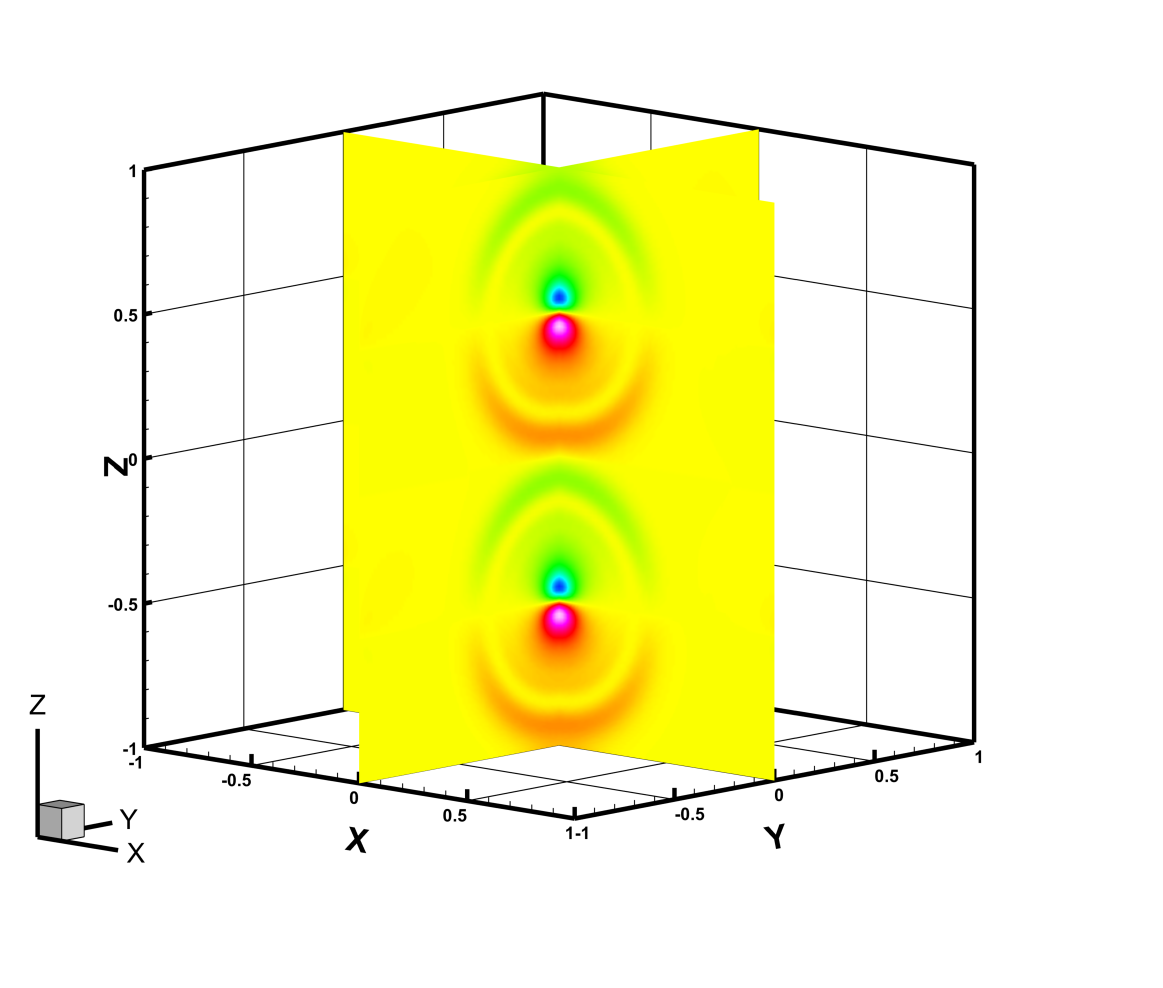}} \hspace{-8pt}
  \subfigure[$t=0.8$ ]{ \includegraphics[scale=.16]{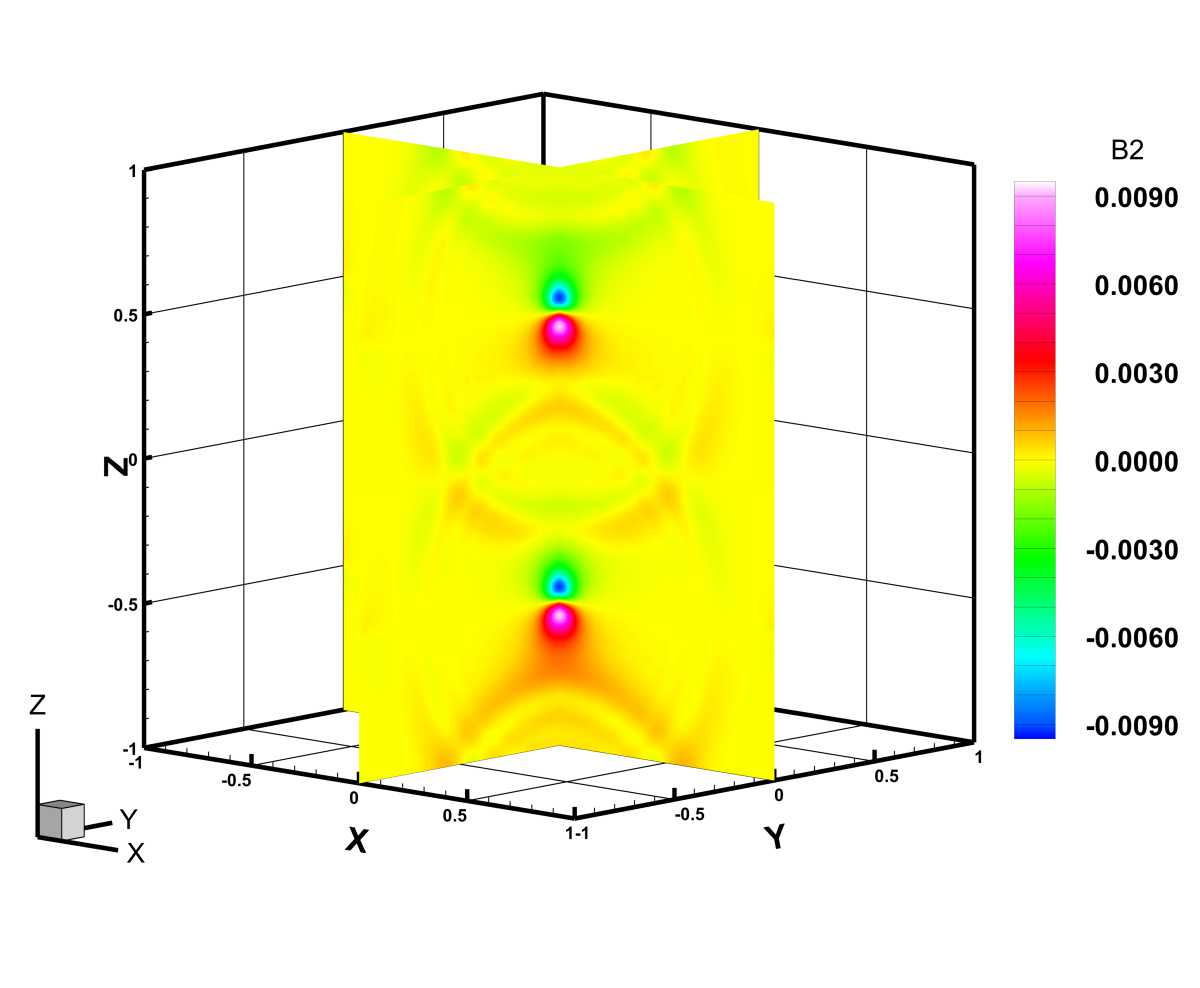}}
 \caption{\small Example \ref{ex: ex6}: electromagnetic wave propagation induced by point source in 3D. (a)-(d) snapshots of $ B_2$ at $t=0.1,0.3,0.46,0.8$. The simulations are obtained with $N=100$ and $\tau=0.02.$} 
   \label{figsource3d}
\end{center}
\end{figure}

\section{Concluding Remarks}
In this paper, we have developed fast and accurate divergence-free spectral methods using generalized Jacobi polynomials for curl-curl problems in two and three dimensions. The fast solver is integrated with two critical components: (i). the matrix-free preconditioned Krylov subspace iterative method; and (ii) the efficient preconditioner driven by a fully diagonalizable auxiliary problem. With these, the new iteration method finds remarkably efficient in solving the curl-curl problem, where the number of iterations can even be reduced to $0$ for \eqref{eq:system} with the smooth solution. More importantly, we prove rigorously that the dimensions of the invariant subspace of the resultant preconditioned linear system of the divergence-free spectral method for the dominate eigenvalue $1$ is $(N-3)^2$ and $2(N-3)^3$, for two- and three-dimensional problems with $(N-1)^2$ and $2(N-1)^3$ unknowns, respectively. Finally, numerical experiments illustrate the efficiency of the proposed scheme. Future studies along this line will include the extension of the current method to unstructured meshes using spectral-element method for two- and three-dimensional problems.

\begin{appendix}

\section{Proof of Lemma \ref{prop: MS}}\label{AppendixA0}
\renewcommand{\theequation}{A.\arabic{equation}}

The derivation can be divided into three cases with respect to the range of $m$.

\vspace{5pt}
\noindent $ \underline{\text{\bf Case\,1}: 1\leq m\leq 2,\,\, 1\leq n\leq N-1:}$ Let us fix $m=1$ first. By the analytic expression of $M$ in equation \eqref{eq: M} and  the definition of $S$ in equation \eqref{eq: S}, one obtains
\begin{equation*}\begin{split}
& M_{11}=\frac{2}{5},\;\;\;  M_{13}=-\frac{1}{5\sqrt{21}},\\& S_{1n}=\sqrt{\frac{3}{2}}\sqrt{\frac{2n+1}{2}}\int_{-1}^1 L_1'L_n' \,{\rm d}\xi,\;\; \;S_{1n}=\sqrt{\frac{7}{2}}\sqrt{\frac{2n+1}{2}}\int_{-1}^1 L_3'L_n'\,{\rm d}\xi.
\end{split}\end{equation*}
Due to the fact that $M$ is penta-diagonal, one has
\begin{equation}\label{eq: MS1}
(MS)_{1n} = M_{11}S_{1n}+M_{13}S_{3n}=\frac{\sqrt{3(2n+1)}}{5}\int_{-1}^1 \big(L_1'(\xi)-L_3'(\xi)/6  \big )L_n'(\xi)\,{\rm d}\xi.  \end{equation}
By using integration by parts and orthogonality of Legendre polynomials, equation \eqref{eq: MS1} leads to
\begin{equation}\label{MScase11}\begin{split}
(MS)_{1n} &=\frac{\sqrt{3(2n+1)}}{4}\Big\{ [(1-\xi^2)L_n(\xi) ]\big|_{-1}^1 +2\int_{-1}^1 L_1(\xi) L_n(\xi) \,{\rm d}\xi \Big\}\\&= \frac{\sqrt{3(2n+1)}}{3}\delta_{1n}=\delta_{1n}.
\end{split}\end{equation}
Following the same procedure, one can obtain that
\begin{equation}\label{MScase12}
\begin{aligned}
(MS)_{2n}= & M_{22}S_{2n}+M_{24}S_{4n}=\frac{1}{14} \sqrt{\frac{2n+1}{5}} \int_{-1}^1 \Big(\frac{10}{3}L_2'(\xi)-L_4'(\xi)\Big)L_n'(\xi)\,{\rm d}\xi           \\
 = &    \sqrt{\frac{2n+1}{5}}  \frac{5}{2}\int_{-1}^1 L_2(\xi)L_n(\xi)\,{\rm d}\xi=\sqrt{\frac{2n+1}{5}}\delta_{2n}=\delta_{2n}.
 \end{aligned}
\end{equation}

\vspace{5pt}
\noindent $ \underline{\text{\bf Case}\,2: 3\leq m\leq N-3,\,\, 1\leq n\leq N-1:}$
Since $M$ is penta-diagonal, $(MS)_{mn}$ reduces to
\begin{equation}\label{eq: 3term}
(MS)_{mn} =  M_{m(m-2)}S_{(m-2)n}+M_{mm}S_{mn}+M_{m(m+2)}S_{(m+2)n},
\end{equation}
among which 
\begin{equation*}
\begin{aligned}
&M_{m(m-2)}S_{(m-2)n}= \frac{1}{2}\sqrt{\frac{2n+1}{2m+1}} \int_{-1}^1 \Big( \frac{-1}{2m-1}  L_{m-2}'(\xi) \Big)\, L_n'(\xi)\,{\rm d}\xi,\\
&M_{mm}S_{mn}= \frac{1}{2}\sqrt{\frac{2n+1}{2m+1}}\int_{-1}^1 \Big( \frac{1}{2m-1}+\frac{1}{2m+3}  \Big)L_m'(\xi)\, L_n'(\xi)\,{\rm d}\xi,\\
& M_{m(m+2)}S_{(m+2)n}=\frac{1}{2}\sqrt{\frac{2n+1}{2m+1}} \int_{-1}^1 \Big( \frac{-1}{2m+3} L_{m+2}'(\xi) \Big) L_n'(\xi)\,{\rm d}\xi.
\end{aligned}
\end{equation*}
By using the recurrence relation \eqref{eq: reccu}, we have that
\begin{equation*}
 \frac{-1}{2m-1} L_{m-2}'(\xi)+ \Big( \frac{1}{2m-1}+\frac{1}{2m+3}  \Big)L_m'(\xi)+\frac{-1}{2m+3} L_{m+2}'(\xi) =L_{m-1}(\xi)-L_{m+1}(\xi).
\end{equation*}
Thus, it is direct to obtain 
\begin{equation*}\label{eq: MScase2}
\begin{aligned}
&(MS)_{mn}=\frac{1}{2}\sqrt{\frac{2n+1}{2m+1}} \int_{-1}^1\big(L_{m-1}(\xi)-L_{m+1}(\xi) \big)L_n'(\xi)\,{\rm d}\xi\\
&=\frac{1}{2}\sqrt{\frac{2n+1}{2m+1}} \Big\{  \big[ (L_{m-1}(\xi)\hspace{-2pt}-\hspace{-2pt}L_{m+1}(\xi))L_n(\xi) \big]_{-1}^1\hspace{-4pt}+\int_{-1}^1(L_{m+1}(\xi)-L_{m-1}(\xi))' L_n(\xi) \,{\rm d}\xi       \Big\}\\
&=\frac{1}{2}\sqrt{\frac{2n+1}{2m+1}}  \int_{-1}^1 (2m+1)L_m(\xi)L_n(\xi)\,{\rm d}\xi=\sqrt{\frac{2m+1}{2n+1}}\delta_{mn}=\delta_{mn},
\end{aligned}
\end{equation*}
where integration by parts, the recurrence relation \eqref{eq: reccu}, the boundary values of Legendre polynomials in \eqref{eq: bcL}, and the orthogonality relation \eqref{eq: Linner} has been applied. 

\vspace{5pt}
\noindent $ \underline{\text{\bf Case}\, 3: N-2\leq m\leq N-1,\,\, 1\leq n\leq N-1:}$
Fix $m=N-2$ and recall that $M$ is penta-diagonal, one has
\begin{equation*}\label{eq: Msn2}
\begin{aligned}
&(MS)_{(N-2)n} =M_{(N-2)(N-4)}S_{(N-4)n}+M_{(N-2)(N-2)}S_{(N-2)n}\\
=&\Big(  M_{(N-2)(N-4)}S_{(N-4)n}+M_{(N-2)(N-2)}S_{(N-2)n} +M_{(N-2)N}S_{Nn} \Big)-M_{(N-2)N}S_{Nn}\\
                       =& \delta_{(N-2)n}+\frac{1}{4}\sqrt{\frac{2n+1}{2N-3} } \frac{n(n+1)}{2N-1} \Big( 1+(-1)^{n+N}  \Big),
\end{aligned}
\end{equation*}
where equations \eqref{eq: M}-\eqref{eq: S} and the fact that the three-term formula in equation \eqref{eq: 3term} equals to the Dirac-delta function have been used.  Again, by invoking \eqref{eq: M}-\eqref{eq: S}, one obtains that for $m=N-1$,
\begin{equation*}\label{eq: Msn3}
\begin{aligned}
&(MS)_{(N-1)n} =  M_{(N-1)(N-3)}S_{(N-3)n}+M_{(N-1)(N-1)}S_{(N-1)n}\\
=&\Big( M_{(N-1)(N-3)}S_{(N-3)n}+M_{(N-1)(N-1)}S_{(N-1)n} +M_{(N-1)(N+1)}S_{(N+1)n} \Big)\\&-M_{(N-1)(N+1)}S_{(N+1)n} \\
                       =& \delta_{(N-1)n}+\frac{1}{4}\sqrt{\frac{2n+1}{2N-1} } \frac{n(n+1)}{2N+1} \Big( 1+(-1)^{n+N+1}  \Big).
\end{aligned}
\end{equation*}
This ends the proof.

\end{appendix}

\section*{Acknowledgment}
Helpful discussions with Professor Huiyuan Li (Institute of Software, Chinese Academy of Sciences) are gratefully acknowledged.

 \bibliographystyle{plain}
\bibliography{refpapers}


\end{document}